\documentclass[reqno,11pt,oneside]{book}
\usepackage{wrapfig}
\usepackage{amssymb,latexsym,amsmath}     % Standard packages
\usepackage{graphicx}
\usepackage{tikz}
\usepackage{esvect}
\usepackage{hyperref}
\usepackage[utf8]{inputenc}
\usepackage{comment}
\usepackage{xcolor}
\usepackage{amsthm,bm,bbm}
\usepackage{fancyhdr}
\usepackage[toc,page]{appendix}
\usepackage{makeidx}
\usepackage{subcaption} 
\usepackage{lipsum}
\usepackage[export]{adjustbox}
\usepackage[bottom]{footmisc}
\theoremstyle{definition}
\newtheorem{defin}{Definition}[section]
\newtheorem{theorem}{Theorem}
\numberwithin{theorem}{section}
\newtheorem{lemma}{Lemma}
\numberwithin{lemma}{section}

\newtheorem{remark}{Remark}[section]

\numberwithin{cor}{section}

\newcommand{\R}{\mathbb{R}}
\newcommand{\Z}{\mathbb{Z}}
\renewcommand{\div}{\operatorname{div}}
\newcommand{\dist}{\mathrm{dist}}

\makeindex

\begin{document}
%\thispagestyle{empty}
%\par
%\raisebox{-.9\height}{\includegraphics[width=3cm]{uspc}}%
%\hfill
%\raisebox{-.9\height}{\includegraphics[width=1cm]{Logo-P7}}%
%\par

%\lipsum[2]

%\begin{figure}
%\centering
%%\begin{subfigure}[b]{0.07\textwidth}
    %\includegraphics[width=\textwidth, right]{Logo-P7}
%\end{subfigure}
%\centering
%\begin{subfigure}[b]{0.2\textwidth}
 %\includegraphics[width=\textwidth, left]{uspc}
%\end{subfigure}
%\end{figure}
{\large

%\vspace*{0.5cm}

\begin{center}

\vspace*{0.5cm}

{\Large {\bf On the use of tent spaces for solving PDEs: A proof of the Koch-Tataru theorem}}

%\vspace*{1cm}
%Pour obtenir le grade de \ \\[1ex]
%{\bf DOCTEUR de l'UNIVERSIT\'E PIERRE ET MARIE CURIE} \ \\
\vspace*{0.5cm}
by\\
{\Large {\bf Pascal Auscher and Ioann Vasilyev\\ }}

\vspace*{0.5cm}

Lecture notes written by Ioann VASILYEV for the mini-course of Pascal AUSCHER given at the Workshop\\
Inhomogeneous Flows: Asymptotic Models and Interfaces Evolution \\
23-27 September, 2019

\vspace*{0.5cm}
\end{center}}

\chapter*{Abstract}
\addcontentsline{toc}{section}{\numberline{}Abstract}
In these notes we will present (a part of) the parabolic tent spaces theory and then apply it in solving some PDE's originated from the fluid mechanics. In more details, to our most interest are the incompressible homogeneous Navier--Stokes equations. These equations 
have been investigated mathematically for almost one century. Yet, the question of proving well-posedness (i.e. existence, uniqueness and regularity of solutions) lacks satisfactory answer.

 A large part of the known positive results in connection with Navier--Stokes equations are those in which the initial data $u_0$ is supposed to have a small norm in some \emph{critical} or \emph{scaling invariant} functional space.
All those spaces are embedded in the homogeneous Besov space $\dot B^{-1}_{\infty,\infty}.$
A breakthrough was made in the paper~\cite{kochtat} by Koch and Tataru, where the authors showed the existence and the uniqueness of solutions to the Navier--Stokes system in case when the norm $\|u_0\|_{\mathrm{BMO}^{-1}}$ is small enough. The principal goal of these notes is to present in detail a new proof of the theorem by Koch and Tataru on the Navier--Stokes system using the tent spaces theory. We do not mean new in the sense simpler but we hope that after having read these notes, the reader will be convinced that the theory of tent spaces is highly likely to be useful in the study of other equations in fluid mechanics.

These notes are mainly based on the content of the article~\cite{pasdor} by P. Auscher and D. Frey. However, in~\cite{pasdor} the authors deal with a slightly more general system of parabolic equations of Navier--Stokes type. Here we have chosen to write down a self-contained text treating only the relatively easier case of the classical incompressible homogeneous Navier--Stokes equations.

\bigbreak\noindent
 {\bf Keywords~:}
Incompressible homogeneous Navier--Stokes equations, Tent spaces, Hardy spaces.

\chapter{Introduction}
\section{Some history}
Let us more rigorously formulate the main result of these notes. To this end, we first introduce a number of definitions.
\begin{defin}
Let $u$ be a vector-field with components in the Schwartz class $\mathrm S(\R^n)$. Leray's projector $\mathbb P(u)$ is defined by $\mathbb P(u):=u-\nabla \Delta^{-1}(\nabla \cdot u)$.

 In other words, this projector is the operator with the Fourier multiplier matrix symbol $M_{\mathbb P}(\xi):=\left(\delta_{i,j}-\frac{\xi_i \xi_j}{|\xi|^2}\right)_{1\leq i,j\leq n}.$ 
\end{defin}
\begin{remark}
There is an equivalent definition of the operator $\mathbb P$ which is~: 
$$\mathbb P(u):=u+(\mathcal R\otimes \mathcal R)u\quad\hbox{with}\quad
\mathcal R:=(\mathcal R_1,\ldots, \mathcal R_n),$$ where $\mathcal R_i$ for $i=1,\ldots, n$ stands for the Riesz transform in $\R^n$.
\end{remark}

\begin{remark}
Note that if $u$ is a vector-field, then $\div(\mathbb P(u))=0.$
%Let $u$ be a vector filed with components in the Schwartz class $\mathrm S(\R^n)$
\end{remark}

We are ready to introduce the incompressible homogeneous Navier--Stokes equations:
\begin{equation}
\label{inhom}
\begin{cases}
\partial_t u + (u \cdot \nabla) u -\Delta u + \nabla p=0,\\
\div u=0, \\ u(0,\cdot) = u_0(\cdot).
\end{cases}
\end{equation}
Here is the physical meaning of the terms above : $p: \mathbb R^n\rightarrow \mathbb R$ is the pressure of the ``ideal'' fluid, $u: \mathbb R^n \rightarrow \mathbb R^n$ is its velocity vector-field and $u_0$ is the initial value of the velocity. The first line in \eqref{inhom} is called momentum equation. The equation $\div u=0$ means that the fluid is incompressible.
\smallbreak
We shall consider in what follows the following \emph{Duhamel's formulation} (also called \emph{mild formulation})
 of the Navier--Stokes equations~:
\begin{equation}
\label{duhmel}
u(t,\cdot)=e^{t\Delta}u_0(\cdot)-\int_0^t e^{(t-s)\Delta} \mathbb P \mathrm{div} (u(s,\cdot)\otimes u(s,\cdot)) \,ds.
\end{equation}
We remind  the reader that for vector-fields $u$ and $v$ the notation $u\otimes v$ stands for the matrix valued function, obtained by multiplying each coordinate function of $u$ by each coordinate function of $v,$ namely  $(u\otimes v)_{i,j}(x):=u_i(x)v_j(x).$ For a matrix valued function $A(x)=(a_1(x),\ldots, a_n(x))^T$ (where $a_j(x)$ is the $j$-th row vector of the matrix $A$), its divergence is defined by $ \mathrm{div}(A):= (\mathrm{div}\, a_1,\ldots, \mathrm{div}\, a_n).$ 

Let us stress that the equation~\eqref{duhmel} is equivalent to the system~\eqref{inhom} under a very mild assumptions on $u_0$. This is proved for instance in~\cite{lemar}. See also \cite{dubois}. The solutions to the system~\eqref{duhmel} are called mild solutions.
\smallbreak
In connection with the system~\eqref{duhmel}, we define the bilinear form $B$ as follows~:
\begin{equation}
B(u,v)(t,\cdot)=\int_0^t e^{(t-s)\Delta} \mathbb P \mathrm{div} (u(s,\cdot)\otimes v(s,\cdot)) \,ds.
\end{equation} 

We are now in position to formulate the Koch and Tataru theorem.
\begin{theorem}
\label{kochtat}
 If the norm $\|u_0\|_{\mathrm{BMO}^{-1}}$ is small enough, then Equations~\eqref{inhom} 
 admit a unique global solution in a ball of the functional space $X$ (see Definition  \ref{prvoX}).
\end{theorem}

We shall give a rigorous definition of the space $\mathrm{BMO}^{-1}$ later 
on\footnote{By $\mathrm{BMO}^{-1}$ here and everywhere after in this text we mean the homogeneous version of this space.}.  For the time being, the reader should think about this space as the set of  distributions $u_0$ such that $u_0=\mathrm{div}(v_0)$ for some vector-field $v_0$ with components in 
the space $\mathrm{BMO}(\R^n).$  
 By $\mathrm{BMO}(\R^n),$ we mean  the space of functions whose mean oscillation is bounded, 
 namely   $\upsilon\in\mathrm{BMO}(\R^n)$ if and only if
$$\sup_{B}\frac{1}{|B|}\int_B|\upsilon(x)-\upsilon_B|\,dx<\infty
\quad\hbox{with}\quad \upsilon_B:=\frac1{|B|}\int_B \upsilon,$$
where $B$ is a  (Euclidean) open ball, $|B|$ is its Lebesgue measure and  the supremum is taken over all  balls $B \subset \R^n$.  The space $\mathrm{BMO}(\R^n)$ is very important in the modern harmonic and Fourier analysis, see for instance the books~\cite{stein} and~\cite{grafmod} and the article~\cite{cora}.
\smallbreak
Several remarks are in order. First, note that the system~\eqref{inhom} is scale invariant under the following parabolic scaling~: 
$$(u,p)(t,x)\leadsto  (\lambda u, \lambda^3p)(\lambda^2t,\lambda x)\quad\hbox{and}\quad
u_0(x)\leadsto \lambda u_0(\lambda x),\qquad\lambda>0.$$
Consequently, one can expect optimal functional spaces for solving this system by means of the Picard fixed point theorem to have the above invariance, for all $\lambda>0$. In particular, `critical' spaces
for initial data have  homogeneity $-1$. Second, the space $\mathrm{BMO}^{-1}$ is contained in the homogeneous Besov space $\dot B^{-1}_{\infty,\infty},$ both spaces are critical for the Navier--Stokes system, and the latter one is the largest critical space. However, it was shown by Bourgain and Pavlovic, see~\cite{bourpav} that the equations~\eqref{inhom} are ill posed in $B^{-1}_{\infty,\infty},$ and this is our third remark here. Finally, we would like to cite some previous results in the direction of the well-posedness of the system~\eqref{inhom}. Cannone (see~\cite{canplan}) considered $u_0\in \dot B^{-1+n/p}_{p,\infty}(\R^n).$ Earlier, Fujita and Kato in \cite{FK} proved well-posedness 
(in the three-dimensional case) for $u_0$ in the Sobolev space $\dot H^{1/2}(\R^3),$ and 
Kato in the paper~\cite{kato} examined the case of small initial data in $L^{n}(\R^n)$ (see also
\cite{GM}).   

Taking into account the  mild formulation \eqref{duhmel}, it does not come up as a surprise that the proof  of well-posedness is based on a fixed point argument. Namely, this will be the Picard contraction principle applied in a very special functional context.
In the following section, we introduce  the  functional spaces in which we will solve 
the Navier-Stokes equations and some related definitions and results. These spaces will allow us to formulate and prove a theorem which is the core of our proof of the Koch and Tataru result. 

%%%%%%%%%%%%%%%%%%%%%%%%%%%

\section{Main definitions and auxiliary results}
Here we collect the most important definitions and results that we shall need afterwards. We begin with the tent spaces.
\begin{defin}
\label{palatprvo}
 Let $ \mathbb R^{n+1}_+:=\bigl\{(t,x),\: t>0\ \hbox{ and }\ x\in\R^n\bigr\}\cdotp$
We shall say that a measurable function $\alpha: \mathbb R^{n+1}_+\rightarrow \mathbb C^n\otimes \mathbb C^n$ belongs to the $(\infty,p)$-parabolic tent space $T^{\infty,p}(\mathbb R^{n+1}_+, \mathbb C^n\otimes \mathbb C^n)$ where $1\le p<\infty$ if
$$\|\alpha\|^p_{T^{\infty,p}(\mathbb R^{n+1}_+, \mathbb C^n\otimes \mathbb C^n)}:=\sup\limits_{x_0\in \mathbb R^n, R>0} \frac{1}{|B(x_0,R)|}\int_{B(x_0,R)\times [0,R^2]}|\alpha(t,x)|^p\, dx\, dt< \infty.$$
By $B(x,R)$ we designate the open Euclidean ball with center $x$ and radius $R$.
\end{defin}
\begin{remark}
In analogy with this definition one can define tent spaces with different target space. Everywhere in this text, we shall simply denote by  $T^{\infty,p}$  the space $T^{\infty,p}(\mathbb R^{n+1}_+, \mathbb C^n\otimes \mathbb C^n),$ unless otherwise specified.
\end{remark}

Tent spaces were first introduced by Coifman, Meyer and Stein in the paper~\cite{cms}, in  the elliptic setting.

\medbreak
Define further the so-called ``path" spaces that will play an important role in what follows.
\begin{defin}
\label{prvoX}
We shall say that a measurable function $u:\mathbb R^{n+1}_+\rightarrow \mathbb C^n\otimes \mathbb C^n$ belongs to the space $X$ if
$$\|u\|_{X}:=\|(t,x)\mapsto t^{1/2}u(t,x)\|_{L_{t,x}^\infty(\mathbb R^{n+1}_+, \mathbb C^n\otimes \mathbb C^n)}+\|u\|_{T^{\infty,2}(\mathbb R^{n+1}_+,\mathbb C^n\otimes \mathbb C^n)}<\infty.$$
Note that $X$ is a Banach space.
\end{defin}
\begin{defin}
%Define a function $\Phi: \mathbb R^n\rightarrow \mathbb R$ by $\Phi(x):= \pi^{-n/2}e^{-|x|^2}$.\
\label{heat}
For $t>0,$ define the heat kernel $\Phi_t: \mathbb{R}^n \rightarrow \mathbb R$, by the following formula $\Phi_t(x):=(4\pi t)^{-n/2}e^{-|x|^2/4t}.$ If we denote by $\Delta$
the Laplace operator in $\R^n,$
 the function given by $U(t,\cdot)=e^{t\Delta}u_0:=\Phi_t\ast u_0$ is a solution of the heat equation 
$$\partial_t U= \Delta U, \qquad U(0,\cdot)=u_0.$$ By definition, we say that $u_0\in \mathrm{BMO}^{-1}$ if $U\in T^{\infty,2}.$
\end{defin}
\begin{remark}
Note that the just given definition of the space  $\mathrm{BMO}^{-1}$ coincides with the one discussed after the formulation of Theorem~\ref{kochtat}. The proof of this fact is contained in the article~\cite{kochtat} and is based on the ``caloric extension'' characterization of the space $\mathrm{BMO}(\R^n)$.
\end{remark}
\begin{defin}
\label{prvoY}
We shall say that a measurable function $\alpha:\mathbb R^{n+1}_+\rightarrow \mathbb C^n\otimes \mathbb C^n$ belongs to the space $Y$ if
$$\|\alpha\|_{Y}:=\|(t,x)\mapsto t\alpha(t,x)\|_{L_{t,x}^\infty(\mathbb R^{n+1}_+,\mathbb C^n\otimes \mathbb C^n)}+\|\alpha\|_{T^{\infty,1}(\mathbb R^{n+1}_+,\mathbb C^n\otimes \mathbb C^n)}<\infty.$$
Note that $Y$ is a Banach space.
\end{defin}

One of the possible approaches to use while trying to solve the Navier--Stokes system is ``separation of time and space''. This  is called the maximal regularity setting.
\begin{defin}\label{defn:maxreg}
 The operator $M^+$ that is defined below is called the maximal regularity operator~:
$$M^+f(t, \cdot):=\int_0^t e^{(t-\tau)\Delta}\Delta f(\tau,\cdot) \,d\tau.$$ 
This is well-definef for $f\in D= L^2(\mathbb R_+^{n+1})\cap L^1(\mathbb R_+, H^2(\mathbb R^n))$.
\end{defin}

Next we state the de Simon theorem, which says that the maximal regularity operator is bounded on $L^2$.
\begin{theorem}
There exists a constant $C$ depending only on $n$ such that for all $f\in D$ holds
$$\|M^+f\|_{L^2(\mathbb R_+^{n+1})} \leq C \|f\|_{L^2(\mathbb R_+^{n+1})}.$$
\end{theorem}
For the reader's convenience, we state a slightly more general result, that is
taken  from the paper~\cite{desimon}.
\begin{theorem}
Let $H$ be a Hilbert space and let $A$ be an operator such that the operator $(-A)$ generates an analytic semigroup bounded by the constant $m_0$. Consider the maximal regularity operator associated with $A,$ namely
$$M_Af(t,\cdot):=\int_0^t e^{-(t-\tau)A}A f(\tau,\cdot) \,d\tau,\qquad t\in{\mathbb R},  \qquad f\in L^2(\mathbb R_+, H) \cap L^1(\mathbb R_+, D(A)).$$
For all $f\in L^2(\mathbb R_+, H) \cap L^1(\mathbb R_+, D(A)),$ there holds
$$\|M_Af\|_{L^2(\mathbb R_+, H)} \leq (m_0+1) \|f\|_{L^2(\mathbb R_+, H)}.$$
\end{theorem}
\begin{proof} Let us extend $f$ by $0$ on ${\mathbb R}_-$ and introduce
$$u(t):= \int_0^t e^{-(t-\tau)A} f(\tau,\cdot)\,d\tau.$$
 For $z\in \mathbb C$ such that $\mathrm{Im}(z)<0,$ the Laplace transform of $u$ reads:
$$\widehat{u}(z):=\int_{\R} e^{- i z t} u(t) \,dt.$$
Note that if $z=\xi+i\eta $ and $\eta<0$ then
\begin{equation*}
\widehat{u}(z):=\int_{\R}\int_{0}^t e^{- i z t} e^{-(t-\tau)A}f(\tau) \,d\tau\, dt =\int_{0}^{+\infty}\int_{\tau}^{+\infty} e^{t(-i z -A) + \tau A} f(\tau) \,dt \,d\tau,
\end{equation*}
where in the second equality we performed a change of variables and used the fact that $f(\tau)=0$ when $\tau<0.$ Since $\eta<0$, the inner integral of the exponential function with respect to the variable $t$  in the right-hand side of the last line above converges. In fact, this integral can be calculated explicitely, giving :
  \begin{equation*}
\widehat{u}(z)=\int_{\R} [- i z -A]^{-1}(-e^{-i z \tau})f(\tau) \,d\tau=[ i z + A]^{-1}\widehat{f}(z).
\end{equation*}

{}From the relation  $z=\xi+i \eta$, we infer the following formula~:
 \begin{equation}
\label{simon1}
 i z \widehat{u}(z)= i z[i z + A]^{-1}\mathcal{F}(e^{ \eta (\cdot)} f(\cdot))(\xi),
\end{equation}
where $\mathcal{F}$ stands for the usual Fourier transform on ${\mathbb R}.$ 
\smallbreak
Since $-A$ generates an analytic semigroup bounded by $m_0,$ we have that
\begin{equation}
\label{simon2}
\| i z[ i z + A]^{-1}\|\leq m_0,
\end{equation}
once again for all $z=\xi+i \eta$ satisfying $\eta<0$.
\medbreak
Observe that if $z=\xi+i \eta$, then~\eqref{simon1} yields
\begin{equation}
\label{simon3}
i z \widehat{u}(z)=\int_{\R}e^{- i z t}u^\prime(t) \,dt=\int_{\R}e^{- i \xi t} e^{ \eta t}u^\prime(t) \,dt = \mathcal{F}(e^{\eta (\cdot)} u^\prime(\cdot))(\xi).
\end{equation}
The last formula follows from one of the basic properties of the Laplace transform (Laplace transform of derivative). 

Fix $\eta<0.$ Since $f\in L^2,$ the Fourier-Plancherel theorem and lines~\eqref{simon1},~\eqref{simon2} and~\eqref{simon3} yield
$$\begin{aligned}
2\pi\|e^{\eta (\cdot)} u^\prime(\cdot)\|_{L^2}\leq \|\mathcal{F}(e^{\eta (\cdot)} u^\prime(\cdot))\|_{L^2}
 &=\|z\widehat u\|_{L^2}\\
&\leq m_0 \|\mathcal{F}(e^{\eta (\cdot)} f(\cdot))\|_{L^2}=2\pi m_0 \|e^{\eta (\cdot)} f(\cdot)\|_{L^2}.
\end{aligned}
$$
To conclude, it suffices to let $\eta$ tend to $0$ and to observe that $M_A f=-u'+ f.$
\end{proof}

Further properties of maximal regularity operators acting on tent spaces can be found for instance in the paper~\cite{amp}.
\medbreak
In the end of this chapter we recall some estimates that concern the Oseen kernel (i.e. the kernel of the operator $e^{t\Delta}\mathbb P$) and its derivatives. The reason why we will need these estimates is that they imply the so-called off-diagonal estimates and the latter are going to be very important for our goals.
\begin{theorem}
Let $\sigma_t$ denote the kernel of the operator $e^{t\Delta}\mathbb P$. For all $\beta\in\mathbb N^n, x\in \R^n$ and $t>0,$ there holds
\begin{equation}
\label{oseen}
|t^{|\beta|/2}\partial^{\beta}_x \sigma_t(x)| \leq C t^{-n/2}\left(1+\frac{|x|}{t^{1/2}}\right)^{-n-|\beta|},
\end{equation}
where $C$ is a constant depending on $n$ and $\beta$ only.
\end{theorem}
The proof of this theorem can be found for example in~\cite[Prop. 11.1]{lemar}.
\medbreak
Finally, let us present  the mentioned above off-diagonal estimates.
\begin{defin}
\label{offdiag}
A family of bounded linear operators $(T_t)_{t>0}$ on $L^2(\R^n)$ is said to satisfy off-diagonal estimates of order $M$, with homogeneity $m$, if there exists a constant $C$ such that for all Borel sets $E,F \subset \R^n$, all $t>0$, and all $f\in L^2(\R^n),$ there holds~:
\begin{equation}
\|\mathbbm{1}_E T_t\mathbbm{1}_F f\|_{L^2}\leq C\left(1+\frac{\mathrm{dist}(E,F)^m}{t}\right)^{-M} \|\mathbbm{1}_F f\|_{L^2}.
\end{equation}
\end{defin}

It is well known that for many differential operators $L$ of order $m$ (such as, for $m=2,$  divergence form elliptic operators with bounded measurable complex coefficients), the family $(tLe^{-tL})_{t\geq 0}$ satisfies off-diagonal estimates of any order, with homogeneity~$m$. This is proved for instance in~\cite{ahlmt}.

\bigbreak Throughout the rest of the text $C$ denotes a ``harmless'' constant. The sign $\lesssim$ indicates that the left-hand part of an inequality is less than the right-hand part multiplied by a constant $C$ as above.
 
 %%%%%%%%%%%%%%%%%%%%%%%%%%%%%%%%%%%%%%
 %%%%%%%%%%%%%%%%%%%%%%%%%%%%%%%%%%%%%%
 
\chapter{Proof of the  Koch and Tataru theorem via tent spaces}
\section{Some preliminary observations}

First of all, let us state the principal result, from which Koch and Tataru's theorem will follow easily via Picard's contraction principle.
\begin{theorem}
\label{zero}
Let $X$ be the Banach space from Definition~\ref{prvoX}. The bilinear operator $B:X\times X\rightarrow X$ is bounded.
\end{theorem}
\begin{proof}
The first observation is that instead of working with the bilinear form $B$, we can consider the linear operator $A:Y\rightarrow X$ (where $Y$ is the Banach space introduced in Definition~\ref{prvoY}) defined by 
$$ A(\alpha)(t,\cdot):=\int_0^te^{(t-s)\Delta}  \mathbb P \mathrm{div} \alpha(s,\cdot)\,ds.$$
Indeed, we have $A(\alpha):=B(u,v)$ for $\alpha=u\otimes v$ and 
it is clear that $(u,v)\mapsto u\otimes v$ maps $X\times X$ to $Y$ since, 
by Cauchy-Schwarz inequality,  
$$\begin{aligned}
\|u\otimes v\|_{Y}&= \| t(u\otimes v)\|_{L^\infty_{t,x}} + \| u\otimes v\|_{T^{\infty,1}}\\
&\lesssim \| t^{1/2}u\|_{L^\infty_{t,x}} \| t^{1/2}v\|_{L^\infty_{t,x}} + \|u\|_{T^{\infty,2}}
\|v\|_{T^{\infty,2}}\\
&\lesssim \|u\|_X\|v\|_X.\end{aligned}$$
Therefore,  Theorem~\ref{zero} just stems from the following  pointwise inequality~:
\begin{equation}
\label{first}
|A(\alpha)(t,x)|\lesssim t^{-1/2}\|\alpha\|_{Y}\quad\hbox{for all }\ (t,x)\in\R^{n+1}_+,
\end{equation}
and the tent space bound~:
\begin{equation}
\label{secondd}
\|A(\alpha)\|_{T^{\infty,2}}\leq \|\alpha\|_{T^{\infty,1}}+\|t^{1/2}\alpha\|_{T^{\infty,2}},
\end{equation}
since, obviously, 
$$\|t^{1/2}\alpha\|_{T^{\infty,2}} \leq \| t\alpha \|_{L^\infty_{t,x}}^{1/2}\|\alpha\|_{T^{\infty,1}}^{1/2}
\leq \|\alpha\|_Y.$$

%Indeed, estimates~\eqref{first} and~\eqref{secondd} will allow us to deduce that
%$$\begin{aligned}\|B(u,v)\|_X&\lesssim \|u\otimes v\|_{T^{\infty,1}}+\|t^{1/2}(u\otimes v)(t,x)\|_{T^{\infty,2}}+\|t(u\otimes v)(t,x)\|_{L^\infty_{t,x}} \\&\lesssim \|u\|_{X} \|v\|_{X},\end{aligned}$$
%where the last estimate in the line above follows from the Cauchy--Schwarz inequality and from the %observation that $|t^{1/2}(u\otimes v)|^2\leq |t^{1/2} u|^2\cdot|v|^2$.

We shall establish the estimate~\eqref{first} (which turns out to be easier) in this section and the estimate~\eqref{secondd} will be proved in the sections 2.2, 2.3 and 2.4. Remark that the quantity $\|t^{1/2}\alpha\|_{T^{\infty,2}}$ is not used in \cite{kochtat}. 
\medbreak
In order to prove  estimate~\eqref{first},   we denote by $k_{\tau}$ the kernel of the operator $e^{\tau\Delta} \mathbb P \div$ and split the integrals in the definition of the operator $A$ as follows~:
\begin{multline}
\label{alinfty}
|A(\alpha)(t,x)|\leq \left|\int_0^t \int_{\R^n\backslash B(x,\sqrt{t})} k_{t-s}(x,y)\alpha(s,y)\, dy \,ds\right| \\
 + \left|\int_{t/2}^t \int_{B(x,\sqrt{t})} k_{t-s}(x,y)\alpha(s,y) \,dy\, ds\right| \\
+ \left|\int_0^{t/2} \int_{B(x,\sqrt{t})} k_{t-s}(x,y)\alpha(s,y) \,dy\, ds\right|\cdotp
\end{multline}
Let us denote by $I_1(t,x),$ $I_2(t,x)$ and $I_3(t,x)$   the three summands above.
As a consequence of  \eqref{oseen}, we have 
$$|k_{t-s}(x,y)|\lesssim (\sqrt{t-s}+|x-y|)^{-n-1}$$ and we shall use this estimate differently in each case. 
\smallbreak
In order to estimate $I_1(t,x),$ consider for all  $i\in \Z^n,$ the points $x_i:=x+i\sqrt{t}$ and the balls $B(x_i,\sqrt{nt})$. Note that these balls cover $\mathbb R^n$.    
%\begin{equation*}
%\label{obedinenie}
%\R^n\backslash \overline{B(x,\sqrt{t})}\subset \bigcup_{i\in \Z^n: |i|>1}B(x_i,\sqrt{t}).\end{equation*}
%that the expression $$\displaystyle{\essup_{\substack{y\in B(x_i,\sqrt{t}),\\ s\in(0,t)}}|k_{t-s}(x,y)|\cdot|B(x_i,\sqrt{t})|}$$ is bounded from above by a constant times $$\displaystyle{\essup_{\substack{y\in B(x_i,\sqrt{t}),\\ s\in(0,t)}}|x-y|^{-n-1}\cdot|B(x_i,\sqrt{t})|}.$$ Hence, the line~\eqref{obedinenie} yields the inequality 
%$$\sum_{i\in \Z^n: |i|>1}\displaystyle{\inf_{\substack{y\in B(x_i,\sqrt{t}),\\ s\in(0,t)}}|x-y|^{-n-1}\cdot|B(x_i,\sqrt{t})|}\lesssim\displaystyle{\int_{\R^n\backslash B(x,(\sqrt{2}-1)\sqrt{t})} |x-y|^{-n-1}\,dy}.$$
%On the other hand, when $z,y\in B(x_i,\sqrt{t})$ for $i$ as above, then $|z-y|\leq \sqrt{t}$ and $|x-z|\geq 2\sqrt{t}.$ Hence in turn $|x-z|/2\leq |x-z|-\sqrt{t}\leq |x-z| - |z-y|\leq |x-y|$. This means that 
%$$\essup_{\substack{y\in B(x_i,\sqrt{t}),\\ s\in(0,t)}}|x-y|^{-n-1} \lesssim \inf_{\substack{y\in B(x_i,\sqrt{t}),\\ s\in(0,t)}}|x-y|^{-n-1}.$$
Using  
$|k_{t-s}(x,y)|\lesssim |x-y|^{-n-1} \le \min((|x-y|^{-n-1}, (\sqrt t)^{-n-1})$ for  $y\in \R^n\backslash {B(x,\sqrt{t})}$ and forgetting about the range of integration, it follows that 
\begin{align}\label{alinfty1}
I_1(t,x)&\leq C \int_0^t \!\!\int_{\mathbb R^n}  \min(|x-y|^{-n-1}, (\sqrt t)^{-n-1}) |\alpha(s,y)| \,dy \,ds\nonumber\\
&\lesssim \sum_{i\in \Z^n}\frac{1}{|B(x_i,\sqrt{nt})|} \int_0^t \!\!\int_{B(x_i,\sqrt{n t})} |\alpha(s,y)| \,dy \,ds \ (nt)^{n/2}(\max(|i|-\sqrt n, 1)t^{1/2})^{-n-1}\nonumber\\
&\lesssim t^{-1/2}\|\alpha\|_{T^{\infty,1}}.
 \end{align}
  
%\begin{eqnarray}
%\label{alinfty1}
%|I_1&&\!\!\!\!\!\!\!\!\!\!(t,x)|\leq C \sum_{i\in \Z^n: |i|>1}\int_0^t \!\!\int_{B(x_i,\sqrt{t})} |x-y|^{-n-1}|\alpha(s,y)| \,dy \,ds \nonumber\\
%C \sum_{i\in \Z^n: |i|>1}\frac{1}{|B(x_i,\sqrt{t})|} \int_0^t \!\!\int_{B(x_i,\sqrt{t})} }|\alpha(s,y)| \,dy \,ds \ t^{n/2}(|i|-1)t^{1/2})^{-n-1}
%&\!\!\!\leq\!\!\!\!& \int_0^t\biggl(\sum_{\substack{i\in \Z^n\\ |i|>1}}\essup_{\substack{y\in B(x_i,\sqrt{t}),\\ s\in(0,t)}}|k_{t-s}(x,y)|\cdot|B(x_i,\sqrt{t})|\biggr)\\[-1ex]
%&&\hspace{5cm}\cdot\sup_{i\in \Z^n}\left(\frac{1}{|B(x_i,\sqrt{t})|}\int_{B(x_i,\sqrt{t})}|\alpha(s,y)|\,dy \right)ds\nonumber\\[1ex]
%&\!\!\!\lesssim\!\!\!\!& \int_0^t\!\left(\int_{\R^n\backslash B(x,(\sqrt{2}-1)\sqrt{t})} |x-y|^{-n-1}\,dy\right) 
%\sup_{z\in \R^n}\left(\frac{1}{|B(z,\sqrt{t})|}\int_{B(z,\sqrt{t})}|\alpha(s,y)|\,dy \right)\! ds\nonumber \\
%&\!\!\!\lesssim\!\!\!\!&  t^{-1/2}\sup_{z\in \R^n}\int_0^t\left(\frac{1}{|B(z,\sqrt{t})|}\int_{B(z,\sqrt{t})}|\alpha(s,y)|\,dy \right) ds\nonumber\\&\!\!\!\lesssim\!\!\!\!& t^{-1/2}\|\alpha\|_{T^{\infty,1}},
%\end{eqnarray}
%where, in the last but one inequality,  we  used the fact that $t\leq1.$ 

\smallbreak
In order to bound $I_3(t,x),$ one can further observe that
 $|k_{t-s}(x,y)|\lesssim |t|^{-(n+1)/2}$ once $s\in (0,t/2)$  and $|x-y|\leq \sqrt{t}$. Hence,
 \begin{eqnarray}
\label{alinfty2}
I_3(x,t) &\!\!\!\lesssim\!\!\!&
\displaystyle \frac{t^{-1/2}}{t^{n/2}}\int_{0}^t \int_{B(x,\sqrt{t})} |\alpha(s,y)| \,dy \,ds \nonumber \\
&\!\!\!\lesssim\!\!\!& \displaystyle t^{-1/2} \|\alpha\|_{T^{\infty,1}}.
\end{eqnarray}
Finally, since we also have  $|k_{t-s}(x,y)|\lesssim |x-y|^{-n+1/2}(t-s)^{-3/4},$ 
one can bound $I_2(t,x)$ as follows~:
 \begin{eqnarray}\label{alinfty3}
I_2(t,x)
 &\!\!\!\lesssim\!\!\!& \biggl(\int_{t/2}^{t}(t-s)^{-3/4}ds\biggr)\biggl(\int_{B(x,\sqrt{t})}|x-y|^{-n+1/2}dy\biggr) \|\alpha\|_{\infty} \nonumber \\
 &\!\!\!\lesssim\!\!\!& t^{1/4}\biggl(\int_0^{\sqrt{t}}\frac{1}{r^{n-1/2}}r^{n-1} dr\biggr) \|\alpha\|_{\infty} \nonumber \\
 &\!\!\!\lesssim\!\!\!& \sqrt{t}\|\alpha\|_{\infty}.
 \end{eqnarray}
 Putting Inequalities~\eqref{alinfty1},~\eqref{alinfty2} and~\eqref{alinfty3} together 
 gives~\eqref{first}.
 \bigbreak
 In order to prove  estimate~\eqref{secondd}, we shall use  the decomposition 
 $$A(\alpha)=A_1(\alpha)+A_2(\alpha)+A_3(\alpha)$$ with 
\begin{equation}
\label{1}
A_1(\alpha)(t,\cdot):=\int_0^t e^{(t-s)\Delta} \Delta (s\Delta)^{-1}(I-e^{2s\Delta})s^{1/2}\mathbb P  \mathrm{div} s^{1/2} \alpha(s,\cdot) \,ds, 
\end{equation}

\begin{equation}
\label{2}
A_2(\alpha)(t,\cdot):=\int_0^\infty e^{(t+s)\Delta} \mathbb P  \mathrm{div} \alpha(s,\cdot) \,ds,
\end{equation}

\begin{equation}
\label{3}
A_3(\alpha)(t,\cdot):=\int_t^\infty e^{(t+s)\Delta} \mathbb P  \mathrm{div} \alpha(s,\cdot) \,ds.
\end{equation}
There are two reasons for such a decomposition. First, $A_1$ can be handled via the maximal regularity techniques. Second, the form of the integral $A_2$ will allow us to use a duality 
argument together with the fact that the Leray projector commutes with the Laplacian. This is where tools from harmonic analysis come into play. Finally, the term $A_3$ should be thought of as a remainder.

%%%%%%%%%%%%%%%%%%%%%%%%%%%%%%%%%s

\section{Estimate of the term \texorpdfstring{$A_1$}{TEXT}}
In this section we shall bound from above the term $A_1.$ Namely, our main goal is to prove the estimate
\begin{equation}
\label{a1}
\|A_1(\alpha)\|_{T^{\infty,2}}\lesssim \|s^{1/2}\alpha\|_{T^{\infty,2}}.
\end{equation} 
First of all, note that $A_1(\alpha)(t,x)=M^+Z(s^{1/2}\alpha)(t,x),$ where $M^+$ is the maximal regularity operator defined in Definition \ref{defn:maxreg} 
and the operator $Z$ in turn is 
defined by $ZF(s,\cdot):=T_sF(s,\cdot)$ with 
$$T_sf:=\mathbb P s^{1/2}\mathrm{div} (s\Delta)^{-1}(I-e^{2s\Delta})f.$$
It is obvious that we are done once we prove that operators $M^+:T^{\infty,2}\rightarrow T^{\infty,2}$ and $Z:T^{\infty,2}\rightarrow T^{\infty,2}$ are bounded. 
\bigbreak
\textbf{Claim 1.} $M^+:T^{\infty,2}\rightarrow T^{\infty,2}$ is a bounded operator.

\begin{proof}
Let $F\in T^{\infty,2}$ and fix $x_0\in \mathbb R^n$ and $R>0.$  
Write $\mathbbm{1}_{[0,R^2]\times\R^n} \,F$ as 
$$\mathbbm{1}_{[0,R^2]\times\R^n} \,F=\sum_{j\geq 0} F_j,$$ where 
$$F_j:=F \mathbbm{1}_{[0,R^2]\times (B(x_0,2^{j+1}R)\backslash B(x_0,2^{j}R))}\ \hbox{ for }\ 
j\geq 1, \ \hbox{ and }\ F_0:=F \mathbbm{1}_{[0,R^2]\times B(x_0, 2R)}.$$ 
We first rule out $M^+F_0$ using de Simon's theorem that ensures that
\begin{equation*}
I_0^2:=\int_{[0,R^2]\times B(x_0,R)}|M^+F_0|^2\,ds\, dx  \leq C \int_{\mathbb R^{n+1}_+}|F_0|^2 
\,ds \,dx.
\end{equation*}
Next, we study the case when $j\geq 1.$  Denote
$$I_j^2:=\int_{[0,R^2]\times B(x_0,R)}|M^+F_j|^2 \,ds \,dx.$$ 
Note that $M^+F_j(t,x)=\sum_{k\geq 1} F_{j,k}(t,x),$ where, for $k\geq 1,$
$$F_{j,k}(t,x)= \int_{t/2^k}^{t/2^{k-1}}e^{(t-s)\Delta}\Delta F_j(s,x) \,ds.$$
 {}From  the triangle inequality, we infer the estimate $$I_j\leq \sum_{k\geq 1}\|F_{j,k}\|_{L^2([0,R^2]\times B(x_0,R))},$$ and hence it suffices to bound from above the $L^2$ norms of the functions $F_{j,k}.$ Note that for $k\geq 2$ one has
$$\begin{aligned}
\|F_{j,k}\|^2_{L^2([0,R^2]\times B(x_0,R))}&=\int_0^{R^2}\int_{B(x_0,R)}\Bigl|\int_{t/2^k}^{t/2^{k-1}}(t-s)e^{(t-s)\Delta}\Delta F_j(s,x)\frac{ds}{t-s}\Bigr|^2\,dx \,dt \\
&\leq\int_0^{R^2}\frac{2^{-k}t}{t^2}\int_{t/2^k}^{t/2^{k-1}}\int_{B(x_0,R)}\Bigl|(t-s)e^{(t-s)\Delta}\Delta F_j(s,x)\Bigr|^2\,dx \,ds\, dt \\
&\leq\int_0^{R^2}\int_{t/2^k}^{t/2^{k-1}}2^{-k}t \left(1+\frac{(2^{j}R)^2}{t-s}\right)^{-2M} \|F_j(s,\cdot)\|^2_{L^2(\R^n)} \,ds \,dt, \end{aligned}
$$
where in the first inequality we used the Cauchy--Schwarz inequality for the integral with respect to $s$ and the Fubini theorem to exchange the integrals in $s$ and in $x$. In the second one we used the off-diagonal estimates for the family $(t-s) e^{(t-s)\Delta}\Delta,$ see Definition~\ref{offdiag}. Note that here we can take $M$ as big as we want. 
\smallbreak
We continue the estimate, now exchanging the integrals in $s$ and in $t$ and also using the fact that $t\lesssim t-s$ for $s$ and $t$ as in the integrals above, getting
$$\begin{aligned}
\|F_{j,k}\|^2_{L^2([0,R^2]\times B(x_0,R))}&\leq \int_0^{R^2/2^{k-1}} \|F_j(s,\cdot)\|^2_{L^2(\R^n)} \int_{s2^{k-1}}^{s2^{k}}2^{-k}t \left(1+\frac{(2^{j}R)^2}{t}\right)^{-2M} \,dt\, ds \\
&\leq 2^{-4Mj} 2^{-k} \|F_j\|^2_{L^2([0,R^2]\times\R^n)}.
\end{aligned}$$
In the case $k=1,$ one can get the  same estimate on the norm 
$\|F_{j,1}\|^2_{L^2([0,R^2]\times B(x_0,R))}$ in an almost identical way as above. Indeed we first obtain the following estimate 
$$\begin{aligned}
\|F_{j,1}\|^2_{L^2([0,R^2]\times B(x_0,R))}&=\int_0^{R^2}\int_{B(x_0,R)}\Bigl|\int_{t/2}^{t}(t-s)e^{(t-s)\Delta}\Delta F_j(s,x)\frac{ds}{t-s}\Bigr|^2\,dx \,dt \\
&\leq\int_0^{R^2}\int_{t/2}^{t}\frac{t}{(t-s)^2}\int_{B(x_0,R)}\Bigl|(t-s)e^{(t-s)\Delta}\Delta F_j(s,x)\Bigr|^2\,dx \,ds\, dt \\
&\leq\int_0^{R^2}\int_{t/2}^{t}\frac{t}{(t-s)^2} \left(1+\frac{(2^{j}R)^2}{t-s}\right)^{-2M} \|F_j(s,\cdot)\|^2_{L^2(\R^n)} \,ds \,dt, \end{aligned}
$$
thanks to the fact that the $L^2$ norm of an integral is bounded from above by the integral of its $L^2$ norm, to the Fubini theorem and to the off-diagonal estimates. Hence, referring to the fact that in this case $t\lesssim s$ we find that
$$\begin{aligned}
\|F_{j,1}\|^2_{L^2([0,R^2]\times B(x_0,R))}&\leq \int_0^{R^2} \|F_j(s,\cdot)\|^2_{L^2(\R^n)} \int_{s}^{2s}\frac{R^2}{(t-s)^2}\left(1+\frac{(2^{j}R)^2}{t-s}\right)^{-2M} \,dt\, ds \\
&\lesssim 2^{-4Mj} \|F_j\|^2_{L^2([0,R^2]\times\R^n)}.
\end{aligned}$$

As a consequence we get for all integer $M$  that $I_j \lesssim 2^{-2Mj} \|F_j\|_{L^2([0,R^2]\times\R^n)}$.  Taking $M>n/4,$ one can thus conclude that~:
$$\begin{aligned}
\biggl(\int_{[0,R^2]\times B(x_0,R)}|M^+F|^2 \,ds\,dx\biggr)^{1/2}&\leq \sum_j I_j \\
&\lesssim R^{n/2}\|F\|_{T^{\infty,2}}\sum_j 2^{(n/2-2M)j} \\
&\lesssim R^{n/2}\|F\|_{T^{\infty,2}},
\end{aligned}$$
which completes the proof of our  claim.
\end{proof}
\bigbreak
\textbf{Claim 2.} $Z:T^{\infty,2}\rightarrow T^{\infty,2}$ is a bounded operator. 
\begin{proof}
Note that $T_sF$ is an integral operator and denote by $\kappa_s$ its kernel. 
In order to achieve our goal,  we need the following  estimate on the function $\kappa_s$ 
(the proof of  which is similar to that of~\eqref{oseen})~:
\begin{lemma} There exists $C>0$ such that for all $(s,x) \in \mathbb R^{n+1}_+$ with $ |x|\geq s^{1/2},$ we have
$$|\kappa_s(x)|\leq Cs^{-n/2}\left(\frac{|x|}{s^{1/2}}\right)^{-n-1}.$$
\end{lemma}
Now we can easily prove the following important property of the family $T_s,$ which will be used in a moment.
%is called the $L^{\infty}-L^2$ off--diagonal estimate.
\begin{lemma}
\label{linftyltwo}
Let $F_j$ be as above and let $s>0$. Then,
$$\|T_sF_j(s,\cdot)\|_{L^\infty(B(x_0,R))} \lesssim s^{-n/4}\left(\frac{s^{1/2}}{2^jR}\right)^{n/2+1}\|F_j(s,\cdot)\|_{L^2(\mathbb R^n)}.$$
\end{lemma}

\begin{proof}
If $x\in B(x_0,R),$ then we deduce using the previous lemma that
\begin{multline*}
\Bigl|\int_{\mathbb R^n}\kappa_s(x-y)F_j(s,y)\,dy\Bigr|\lesssim s^{-n/2}\biggl(\frac{2^jR}{s^{1/2}}\biggr)^{-n-1}\int_{B(x_0,2^{j+1}R)\backslash B(x_0,2^{j}R)}|F(s,y)|\,dy\\
\lesssim s^{-n/2}\left(\frac{2^jR}{s^{1/2}}\right)^{-n-1}(2^jR)^{n/2}\left(\int_{B(x_0,2^{j+1}R)\backslash B(x_0,2^{j}R)} |F(s,y)|^2\,dy\right)^{1/2},
\end{multline*}
where in the second estimate we have used the H\"older inequality. Hence Lemma~\ref{linftyltwo} follows. 
\end{proof}

As in the previous claim, we first concentrate on the function $F_0$. We use the fact that the operators $T_s$ are uniformly bounded in $L^2$ with respect to  $s$ so as to write: 
\begin{eqnarray}
\label{ZF0}
\int_{[0,R^2]\times B(x_0,R)} |ZF_0|^2 \,ds\,dx&\!\!\!\leq\!\!\!& \int_{[0,R^2]\times\mathbb R^n} |T_sF_0|^2 \,ds\,dx\nonumber \\
&\!\!\!\leq\!\!\!& C \int_{[0,R^2]\times \mathbb R^n} |F_0|^2 \,ds\,dx
=C \int_{[0,R^2]\times B(x_0,R)} |F|^2 \,ds\,dx\qquad\nonumber\\&\!\!\! \leq\!\!\!& C \|F\|^2_{T^{\infty,2}} |B(x_0,R)|.
\end{eqnarray}

Next we turn to the off-diagonal terms, i.e. we consider indices $j\in \mathbb N$ such that $j\geq 1.$ Fix $s>0$. The H\"older inequality and Lemma~\ref{linftyltwo} imply that
\begin{eqnarray}
\|T_sF_j(s,\cdot)\|_{L^2(B(x_0,R))}&\!\!\!\leq\!\!\!& R^{n/2}\|T_sF_j(s,\cdot)\|_{L^\infty(B(x_0,R))}
\nonumber \\
&\!\!\!\leq\!\!\!& C R^{n/2}s^{-n/4}\left(\frac{s^{1/2}}{2^jR}\right)^{n/2+1}\|F_j(s,\cdot)\|_{L^2(\mathbb R^n)}.
\end{eqnarray}
Hence, from  the fact that $s^{1/2}\leq R,$ we infer the estimates
\begin{eqnarray}
\label{TsFj}
\frac{1}{R^n}\int_{(0,R^2)\times B(x_0,R))}|T_s F_j|^2 \,ds\,dx 
&\!\!\!\leq\!\!\!& \frac{2^{-2j}}{(2^jR)^n} \int_{(0,R^2)\times\left(B(x_0,2^{j+1}R)\backslash B(x_0,2^{j}R)\right)}|T_s F_j|^2 \,ds\,dx\nonumber \\
&\!\!\!\lesssim\!\!\!& 2^{-2j} \frac{1}{|B(x_0,2^{j+1}R)|} 
\int_{(0,2^{j+1}R^2)\times B(x_0,2^{j+1}R)}|F|^2 \,ds\, dx\quad \nonumber \\
&\!\!\!\lesssim\!\!\!& 2^{-2j} \|F\|^2_{T^{\infty,2}}.
\end{eqnarray}
Putting he lines~\eqref{ZF0} and~\eqref{TsFj} together  completes  the proof the second claim.
\end{proof}

As pointed out above, estimate~\eqref{a1} follows obviously from the first and the second claims.

%%%%%%%%%%%%%%%%%%%%%%%%%%

\section{Estimate of the term \texorpdfstring{$A_3$}{TEXT}}

Our next goal is to prove 
\begin{equation}
\label{a3}
\|A_3(\alpha)\|_{T^{\infty,2}}\lesssim \|s^{1/2}\alpha\|_{T^{\infty,2}}.
\end{equation} 

Recall  that the operator $A_3$ is defined by
$$A_3\alpha(t,\cdot)=\int_t^\infty e^{(t+s)\Delta} \mathbb P \div \alpha(s,\cdot)\,ds.$$
Observe that $A_3(\alpha)=\mathcal R(s^{1/2}\alpha)$, where 
$$(\mathcal R F)(t,\cdot):=\int_t^\infty K_{t,s}F(s,\cdot)\,ds,$$
and the operator $K_{t,s}$ is defined by $K_{t,s}:=e^{(t+s)\Delta}\mathbb P s^{-1/2} \div$ for $s,t>0.$
\smallbreak
Note that  $K_{t,s}$ is a kernel operator for all $s,t>0,$ with   kernel  $k_{t,s}$ satisfying,  according to~\eqref{oseen},
\begin{equation}
\label{yadroT_s}
|k_{t,s}(x)|\leq C s^{-1/2}(t+s)^{-1/2-n/2}(1+(t+s)^{-1/2}|x|)^{-n-1},
\end{equation}
for all $x\in \R^n$ and $s,t>0.$
\medbreak
We shall first prove that $\mathcal R$ is a bounded operator on the space $L^2(\R^{n+1}_+).$ This follows from 
\begin{equation}\label{yadroT_s2}
\|K_{t,s}\|_{L^2(\R^n)\rightarrow L^2(\R^n)}\leq Cs^{-1/2}(t+s)^{-1/2},\end{equation} which, in turn,  is an easy consequence of the estimate~\eqref{yadroT_s}. 
\smallbreak
In order to prove the boundedness of ${\mathcal R},$  pick some $\beta\in (-1/2,0),$ set $p(t)=t^{\beta}$ and observe that the function $k(t,s):=\|K_{t,s}\|_{L^2\rightarrow L^2}\mathbbm{1}_{(t,\infty)}(s)$ satisfies
$$\begin{aligned}\int_0^\infty k(t,s)p(t) \,dt &\lesssim \int_0^s s^{-1/2} t^{-1/2} t^{\beta} \,dt \lesssim s^{\beta}=p(s), &\text{ for all } s>0,\\
\int_0^\infty k(t,s)p(s) \,ds &\lesssim \int_t^\infty s^{-1/2} s^{-1/2} s^{\beta} \,ds \lesssim t^{\beta}=p(t), &\text{ for all } t>0.\end{aligned}$$ 
The desired $L^2(\R^{n+1}_+)$ boundedness now follows since,  applying the Minkowski inequality and the Schur test, we have :
$$\begin{aligned}\|\mathcal R F\|^2_{L^2(\R^{n+1}_+)}&=\int_{\R_+}\int_{\R^n}\left|\int_t^\infty K_{t,s}F(s,x)\,ds\right|^2 \,dx \,dt  \\ &\leq \int_{\R_+}  \biggl(\int_{t}^\infty \Bigl(\int_{\R^n} |K_{t,s}F(s,x)|^2\,dx\Bigr)^{1/2}\,ds\biggr)^2 dt  \\ &\lesssim \int_{\R_+}\biggl(\int_{\R_+} k(t,s) \|F(s,\cdot)\|_{L^2(\R^n)} \,ds\biggr)^2 dt\\
 &\lesssim \|F\|^2_{L^2(\R^{n+1}_+)}. \end{aligned}$$

Next, we concentrate on the boundedness of $\mathcal R$ on $T^{\infty,2}.$
As a first, observe that Inequality  \eqref{yadroT_s}   readily implies the following $L^2-L^\infty$ off-diagonal estimate
for all disjoint Borel sets $E, \tilde{E} \subseteq \R^n$ and $s,t>0$~:
\begin{equation}
\label{ltwolinfty}
\|\mathbbm{1}_E K_{t,s} \mathbbm{1}_{\tilde{E}}\|_{L^2(\R^n)\rightarrow L^\infty(\R^n)}\!\lesssim\!  s^{-1/2}(t\!+\!s)^{-1/2-n/4}\left(1+(t\!+\!s)^{-1/2}\dist(E,\tilde{E})\right)^{-n/2-1}.
\end{equation}
Let $F\in T^{\infty,2}$ and fix $(R,x_0)\in \R^{n+1}_+.$ Define $B_j:=(0,2^jR^2)\times B(x_0,2^jR)$ for $j\geq 0$ and $C_j:=B_j\backslash B_{j-1}$ for $j\geq 1$. Then, set $F_0:=\mathbbm{1}_{B_0} F$ and  $F_j:=\mathbbm{1}_{C_j}F$ for $j\geq 1$. Using the Minkowski inequality we deduce that
\begin{multline*}
\biggl(R^{-n}\int_0^{R^2}\|(\mathcal RF)(t,\cdot)\|^2_{L^2(B(x_0,R))}\,dt\biggr)^{1/2}\\
\lesssim \sum_{j\geq 0} \biggl(R^{-n}\int_0^{R^2} \|(\mathcal RF_j)(t,\cdot)\|^2_{L^2(B(x_0,R))}\,dt\biggr)^{1/2}=:\sum_{j\geq 0} I_j.
\end{multline*}
For a natural number $j$ such that $j\leq 2,$ the boundedness of $\mathcal R$ on $L^2$ yields the estimate $I_j\lesssim\|F\|_{T^{\infty,2}}.$ For $j\geq 3$, split $C_j$ as follows~:
\begin{multline*}
C_j=(0,2^{j-1}R^2)\times (B(x_0,2^jR)\backslash B(x_0,2^{j-1}R)) \\
\cup (2^{j-1}R^2,2^jR^2)\times B(x_0,2^jR) =:C_j^{(0)}\cup C_j^{(1)}.
\end{multline*}
 Denote $F_j^{(0)}:=\mathbbm{1}_{C_j^{(0)}}F$ and $F_j^{(1)}:=\mathbbm{1}_{C_j^{(1)}}F$ and, correspondingly, $I_j^{(0)}$ and $I_j^{(1)}$.

For $I_j^{(0)}$, split the integral in $s$ and use the H\"older inequality to obtain
$$I_j^{(0)}\lesssim \sum_{k\geq 0}\biggl(R^{-n}\int_0^{R^2} \int_{2^kt}^{2^{k+1}t} 2^kt \|K_{t,s}F_j^{(0)}(s,\cdot)\|^2_{L^2(B(x_0,R))}\,ds\, dt\biggr)^{1/2}.$$
Now observe that for  $j\geq 3, k\geq 0, t\in (0,R^2)$ and $s\in(2^kt,2^{k+1}t),$ H\"older's inequality and the $L^2-L^{\infty}$ off-diagonal estimate~\eqref{ltwolinfty} above yield for any $\delta\in (0,1]$:
$$\begin{aligned}
\|K_{t,s}F_j^{(0)}(s,\cdot)&\|_{L^2(B(x_0,R))} \lesssim R^{n/2} \|K_{t,s}F_j^{(0)}(s,\cdot)\|_{L^\infty(B(x_0,R))} \\
&\lesssim R^{n/2}s^{-1/2} (t\!+\!s)^{-1/2-n/4} \biggl(1+\frac{2^{j-1}R-R}{(t\!+\!s)^{1/2}}\biggr)^{-n/2+\delta}\|F_j(s,\cdot)\|_{L^2}\\
&\lesssim (2^j)^{-n/4-\delta/2}R^{-\delta}(2^kt)^{-1+\delta/2}\|F_j(s,\cdot)\|_{L^2}.
\end{aligned}$$
Combining this estimate with the previous one, interchanging the order of integration and choosing $\delta<1$ gives
$$\begin{aligned}
\sum_{j\geq 1}I_j^{(0)} &\lesssim \sum_{j\geq 1}\sum_{k\geq 0} 2^{-j\delta/2-k(1/2-\delta/2)} \biggl((2^jR^2)^{-n/2} \int_0^{2^jR^2} \|F_j(s,\cdot)\|^2_{L^2}\,ds\biggr)^{1/2} \\ &\lesssim \|F\|_{T^{\infty,2}}.
\end{aligned}$$

For $I_j^{(1)}$ it is enough to use the $L^2-L^\infty$ bound for the operator $K_{t,s}$ instead of the off-diagonal estimates. For $s\in (2^{j-1}R^2,2^{j}R^2)$ and $0<t< R^2$ one thus obtains
$$\begin{aligned}
\|K_{t,s}F_j^{(1)}(s,\cdot)\|^2_{L^2(B(x_0,R))} &\lesssim R^{n/2} \|K_{t,s}F_j^{(1)}(s,\cdot)\|^2_{L^\infty(B(x_0,R))} \\
&\lesssim R^{n/2}s^{-1/2} (t+s)^{-1/2-n/4} \|F_j(s,\cdot)\|_{L^2}\\
&\lesssim (2^j)^{-n/4}(2^j R^2)^{-1}\|F_j(s,\cdot)\|_{L^2}.
\end{aligned}$$
Plugging this into $I_j^{(1)}$ yields 
$$\begin{aligned}
I_j^{(1)}&\lesssim  \biggl(R^{-n}\int_0^{R^2} \int_{2^jt}^{2^{j+1}t} 2^jR^2 \|K_{t,s}F_j^{(1)}(s,\cdot)\|^2_{L^2(B(x_0,R))\,}ds \,dt\biggr)^{1/2}\\
&\lesssim \biggl((2^jR^2)^{-n/2} \int_0^{2^jR^2} (2^jR^2)^{-1/2}R\|F_j(s,\cdot)\|^2_{L^2}\,ds\biggr)^{1/2}\\&\lesssim 2^{-j/2}\|F\|_{T^{\infty,2}}.
\end{aligned}$$
Summing up over $j$ gives $\|\mathcal R F\|_{T^{\infty,2}}\lesssim \|F\|_{T^{\infty,2}},$ whence \eqref{a3} is proved.

%%%%%%%%%%%%%%%%%%%%%%%%%%%

\section{Estimate of the term  \texorpdfstring{$A_2$}{TEXT}}

We now wish to prove 
\begin{equation}
\label{a2}
\|A_2(\alpha)\|_{T^{\infty,2}}\lesssim \|\alpha\|_{T^{\infty,1}}.
\end{equation}

This is where we need tools coming from harmonic analysis.  Let us first introduce  some terminology that is very important in this section.
\begin{defin}
We shall say that a continuous function $u:\mathbb R_+^{n+1}\to \mathbb C$ belongs to the tent space $T^{1,\infty}$ if
$\|u\|_{T^{1,\infty}}:=\|N(u)\|_{L^1(\mathbb R^{n})}<\infty,$
where
 $$N(u)(x):=\sup_{\{(t,y) : y\in B(x,\sqrt{t})\}}|u(t,y)|,$$ is the parabolic non-tangential maximal function  and $$\lim_{t\to 0, \,  y\in B(x,\sqrt{t})  }u(t,y)\ 
\text{exists for almost every} \ x\in \mathbb R^n.$$ 
\end{defin}
\begin{defin}
We shall say that a measurable function $u$ belongs to the tent space~$T^{1,2}$ if
$\|u\|_{T^{1,2}}:=\|S(u)\|_{L^1(\mathbb R^{n})}<\infty,$
where $$S(u)(x):=\left(\iint_{\{(t,y):y\in B(x,\sqrt{t})\}}|u(t,y)|^2\,\frac{dy\, dt}{t^{n/2}}\right)^{1/2}$$
is the parabolic square function.
\end{defin}
\begin{defin}
Let $x\in \mathbb R^n.$ The parabolic cone $\Gamma(x)$ 
with vertex $x$ is defined by $$\Gamma(x):=\bigl\{(t,y): y\in B(x,\sqrt{t})\bigr\}\cdotp$$ 
\end{defin}
\begin{defin}
Let $O\subset \mathbb R^n$ be an open set. The tent $\widehat{O}$ (also called the parabolic tent above the set $O$) is defined by $\widehat{O}:=\{(t,y): \mathrm{dist}(y,O^c)\geq \sqrt{t}\}$
(the tent above a ball is pictured in red in the figure just below):  %The point $x$ is called the vertex of the cone $\Gamma(x).$
\end{defin}
\begin{tikzpicture}
\draw[help lines, color=gray!90, dashed] (-1,-1) grid (5.9,2.9);
\draw[->,thick] (-1,0)--(6,0) node[right]{$(x,\R^1)$};
\draw[->,thick] (0,-1)--(0,3) node[above]{$(t,\R_+)$};
%\draw[thin,red] (2,0) arc (270:360:1cm);
%\draw[thin,red] (3,1) arc (180:270:1cm);
\filldraw[fill=red!90!white, draw=red!90!black]
    (1,0) arc (270:360:2cm) -- (3,2) arc (180:270:2cm) -- (1,0) ;%-- cycle;
\end{tikzpicture}
  \begin{defin}
A Borel measure $\mu$ on $\mathbb R^{n+1}_+$ is called a Carleson measure if $$\sup_{B \subset \mathbb R^n -\text{open ball}} \frac{\mu(\widehat{B})}{|B|}<\infty.$$
\end{defin}
\noindent
Let $B(x_0,r)\subset \R^n$ be an open ball. Note that $$ \Bigl[0,\left(\frac{r}{2}\right)^2\Bigr]\times B\left(x_0,\frac{r}{2}\right)  \subset\widehat{B(x_0,r)}\subset [0,r^2]\times B(x_0,r). $$
Thus (recall Definition~\ref{palatprvo}) $F\in T^{\infty,1}$ if and only if $|F| \,dy \,dt$ is a Carleson measure.
\medbreak
Introduce for $x\in \mathbb R^n$ the following important function: $$\mathcal C(\mu)(x):=\sup_{x\in B, B \\-\text{an open ball}}\frac{\mu(\widehat{B})}{|B|}$$ and the corresponding norm $\|\mu\|_{\mathcal C}:=\|\mathcal C(\mu)\|_\infty.$ We shall further write $\mathcal C(F)$ instead of $\mathcal C(|F|\,dy\, dt)$ in the case where $F$ is a function.
 \medbreak
The strategy to  prove the boundedness  $A_2: T^{\infty,1} \rightarrow T^{\infty,2}$ comprises three main steps.
The first one is the embedding property $T^{\infty,1}\subset (T^{1,\infty})^*$, the second one is the fact that $T^{\infty,2}= (T^{1,2})^*$ and the third one is that $A_2^*: T^{1,2}\rightarrow T^{1,\infty},$ where $A_2^*(G)(s,\cdot):=e^{s\Delta}\int_0^\infty \nabla \mathbb P e^{t\Delta} G(t,\cdot)\, dt.$ Indeed, suppose that these steps are proved. Then, for $F\in T^{\infty,1}$  and $G\in  T^{1,2}$,  using
$$
  \int_{\R^{n+1}_+} A_2(F)  \overline{G} \,dy \,dt = \int_{\R^{n+1}_+} F \overline{ A_2^*(G)} \,dy \,dt
  $$
  we deduce that 
  $$\|A_2(F)\|_{T^{\infty,2}} \le \sup_{\|G\|_{T^{1,2}}=1} \|F\|_{T^{\infty,1}} \|A_2^*(G)\|_{T^{1,\infty}} \lesssim   \|F\|_{T^{\infty,1}}.
  $$
\medbreak
\textbf{Step 1.}   The  main  idea here is to use the Carleson embedding:
\begin{equation}
\label{carlemb}
\int_{\R^{n+1}_+} |H|\,d\mu \leq C_n \int_{\R^n} N(H)(x) C(\mu)(x) \,dx,
\end{equation}
whenever $H\in T^{1,\infty}$ and $\mu$ is a Carleson measure for some constant $C_n$ depending only on the dimension $n$ of the ambient space. 
\smallbreak
Let us explain how the first step follows from~\eqref{carlemb}. Indeed, if $F\in T^{\infty,1}$ and $d\mu:=|F|\,dy \,dt,$ then the mapping  $H\mapsto \int H F \,dy \,dt$
is a bounded functional on  $T^{1,\infty}.$  
\smallbreak
So, let us  focus on the proof of inequality~\eqref{carlemb}. Since $N(H)$ is lower semi-continuous with  $N(H)\in L^1(\R^{n}),$ for $\lambda>0$ the set $O_\lambda=\{x\in \R^n : N(H)(x)>\lambda\}$ is open with $|O_\lambda|<\infty$ (and hence $O_\lambda\subsetneq \R^n$). Consider the Whitney decomposition of the set $O_\lambda$: there exists a constant $c_n\in (0,1)$ and a sequence $(B_j)_j$ of open balls contained in $O_\lambda$ such that $O_\lambda \subset \bigcup B_j, 4B_j \not\subset O_\lambda$ and the balls~$c_n B_j$ are mutually disjoint.  We want an estimate on the value $\mu(\{(t,y):|H(t,y)|>\lambda\}).$ Note that $|H(t,y)|>\lambda \Rightarrow B(y,\sqrt{t})\subset O_\lambda \Rightarrow (t,y)\subset \widehat{O_\lambda}.$ Hence,
$$\mu(\{(t,y):|H(t,y)|>\lambda\})\leq \mu(\widehat{O_\lambda}).$$
Since $(t,y)\in \widehat{O_\lambda},$ there exists $i$ such that $y\in B_i(=:B(x,r_i))$ and there exists $z\not\in O_\lambda$ such that $z\in 4B_i.$ It follows that
$$\sqrt{t}\leq \dist(y,O_\lambda^c)\leq |y-z|\leq |y-x_i|+|x_i-z|\leq r_i+4r_i=5r_i\leq \dist(y,\bigcup_i 6B_i).$$
Hence, $(t,y)\in \widehat{6B_i}.$ Now, we see that there holds
$$\begin{aligned}
\mu(\widehat{O_\lambda})\leq \sum_i \mu(\widehat{6B_i}) &\leq \sum_i|6B_i|
\inf_{x\in 6B_i} \mathcal C(\mu)(x) \\
&\leq \left(\frac{6}{c_n}\right)^n \sum_i|c_nB_i|\inf_{x\in c_nB_i}\mathcal C(\mu)(x)\\
&\leq \left(\frac{6}{c_n}\right)^n\int_{\bigcup_i c_n B_i}\mathcal C(\mu)(x)\, dx \\
&\leq \left(\frac{6}{c_n}\right)^n \int_{O_\lambda} \mathcal C(\mu)(x) \,dx.
\end{aligned}
$$
This allows us to complete  the proof of the estimate~\eqref{carlemb}  by integrating in $\lambda>0$ and  thus that 
of  the first step.
\bigbreak
\textbf{Step 2.} The  key to  the proof of the embedding $T^{\infty,2}\subset (T^{1,2})^*$  is the following estimate~:
\begin{equation}
\label{second}
\int_{\R^{n+1}_+} |F||G| \,dy \,dt\leq C \int_{\R^n} \mathcal C_2(F)(x) S(G)(x) \,dx,
\end{equation}
where $\mathcal C_2(F):=(\mathcal C(|F|^2))^{1/2}$, for all  $F\in T^{\infty,2}$ and $G\in T^{1,2}$. Indeed, the inequality~\eqref{second} obviously yields $$\left|\int_{\R^{n+1}_+} FG \,dy \,dt\right|\lesssim \|C_2(F)\|_{L^\infty(\R^n)}  \|S(G)\|_{L^1(\R^n)}= \|F\|_{T^{\infty,2}}  \|G\|_{T^{1,2}},$$
which realizes  $F$ as a bounded linear functional on the space $T^{1,2}.$
\begin{remark}
Note that an easier inequality 
\begin{equation*}
\int_{\R^{n+1}_+} |F||G| \,dy \,dt\leq  \int_{\R^n} S(F)(x) S(G)(x)\, dx
\end{equation*}
follows from the Cauchy--Schwarz inequality. Alas, $S(F)$ is not comparable to $\mathcal C_2(F)$ when $F\in T^{\infty,2}.$
\end{remark}
 
In order to prove~\eqref{second} we shall use a stopping time argument. To this end, we recall a couple of definitions.
\begin{defin}
Let $h>0$ and let $x\in \R^n.$ The truncated cone $\Gamma_h(x)$ is defined by 
$$\Gamma_h(x):=\bigl\{(t,y)\in \mathbb R^{n+1}_+: y\in B(x,\sqrt{t}) \text{ and } \sqrt{t}\leq h\bigr\}\cdotp$$ The corresponding average is given by
$$S_h(G)(x):=\left(\int_{\Gamma_h(x)}|G|^2\, \frac{dy\, dt}{t^{n/2}}\right)^{1/2}.$$
\end{defin} 
We also need to define the function $\mathfrak h$ on $\R^n$ by the formula $$\mathfrak h(x):=\sup\{h>0: S_h(F)(x)\leq \nu \mathcal C_2 (F)(x)\},$$ where $\nu$ is a large universal constant to be disclosed in a moment ($\nu=3^n\cdot100$ will work). 
\medbreak
Let $\phi : \mathbb R^{n+1}_+\rightarrow \mathbb R$ be a nonnegative function. The Fubini theorem yields the formula
\begin{equation}
\label{fubini}
I:=\int_{\R^n}\biggl(\int_{\Gamma_{\mathfrak h(x)}(x)}\phi(t,y) dy dt\biggr)dx=\int_{\R^{n+1}_+}\phi(t,y)\biggl(\int_{x\in B(y,\sqrt{t}), \sqrt t\le \mathfrak h(x) } \,dx\biggr)dy\,dt.
\end{equation}
The function $\phi$ will be chosen in a moment. Denote until the end of this step $B:=B(y,\sqrt{t}).$
Consider the set $E:=\{x\in B(y,\sqrt{t}): \mathfrak h(x)\geq \sqrt{t}\}.$  Suppose that we know that $|E|>c|B|$ for some constant $c\in (0,1)$ independent of $y$ and $t$. How to conclude then~? Indeed, if it is the case, then choose $\phi:=|F||G|/\sqrt{t^n}$ and note that~\eqref{fubini} gives $I\geq c\int_{\R^{n+1}_+} \phi(t,y)\sqrt{t^n}\,dy \,dt,$  whence
$$\begin{aligned}
\int_{\R^{n+1}_+}|F||G| \,dy \,dt &\leq \frac{1}{c} \int_{\R^n} \left(\int_{\Gamma_{\mathfrak h(x)}(x)}|F||G|\,\frac{dy\, dt}{t^{n/2}}\right)dx \\
&\leq \frac{1}{c} \int_{\R^n}S_{\mathfrak h(x)}(F)(x)S_{\mathfrak h(x)}(G)(x)\,dx\\ &\leq \frac{\nu}{c} \int_{\R^n} \mathcal C_2(F)(x)S(G)(x) \,dx.
\end{aligned}$$
It is left to show that $|E|>c|B|$ for some %$B=B(y,\sqrt{t})$ and 
$c\in (0,1).$ It is sufficient to prove that $|B\backslash E|\leq (1-c)|B|.$  Take a point $x\in B\setminus E.$ As $x\in E^c$,  $\tau:=\sqrt{t}> \mathfrak h(x)$ and $S_{\tau}(F)(x)> \nu \mathcal C_2(F)(x).$ On the other hand,
$$\begin{aligned}
\frac{1}{|B|}\int_B S_\tau^2(F)(x)\, dx &\leq \frac{1}{|B|}\int_B \int_{\Gamma_\tau(x)} |F|^2\,\frac{dy\, dt}{t^{n/2}} \,dx \\
&\leq \frac{1}{|B|}\int_{\widehat{3B}} |F|^2 \,dy \,dt\\
& \leq\frac{|3B|}{|B|} \inf_{x\in 3B} \mathcal C(|F|^2)(x)\\
&\leq {3^n} \inf_{x\in B\backslash E} \mathcal C(|F|^2)(x)\\&\leq \frac{3^n}{\nu\cdot|B\backslash E|}\int_{B\setminus E} S_\tau^2(F)(x) \,dx
\\& \leq \frac{3^n}{\nu\cdot|B\backslash E|}\int_B S_\tau^2(F)(x) \,dx.
\end{aligned}$$
Hence, $\nu \cdot |B\backslash E| \leq 3^n |B|$ and with the choice $\nu=3^n\cdot100$ revealed above, one can take $c=0.99.$
\medbreak
Now, why $T^{\infty,2}\supset (T^{1,2})^*$ ? In order to prove this, we shall need one important type of $T^{1,2}$ functions, called atoms, that are  introduced in the following definition.

\begin{defin}
We say that a function  $a\in T^{1,2}$ is an atom if there exists  a ball $B\subset \R^n$ such that $\mathrm{supp}(a)\subset \widehat{B}$ and $(\int_{\widehat{B}}|a|^2 \,dy\,dt)^{1/2}\leq |B|^{-1/2}.$
\end{defin}

\begin{remark}
One of the main advantages of $T^{1,2}$ atoms is that they are compactly supported. More than that, the $L^2$ norm of an atom is controlled by the size of its support. We shall soon see that any function in the space $T^{1,2}$ admits a representation in terms of  an infinite converging linear combination of $T^{1,2}$ atoms.
\end{remark}

\begin{remark}
Observe that if $a$ is a $T^{1,2}$ atom, then $\|a\|_{T^{1,2}}\lesssim 1.$
\end{remark}
\noindent
So, first of all, note that if a function $f$ is supported in some compact set $K\subset \R^{n+1}_+$ and if $f\in L^2(K)$, then also $f\in T^{1,2}$ with $\|f\|_{T^{1,2}}\leq C(K)\|f\|_{L^2(K)}$. Suppose further that $\ell$ is a bounded linear functional on $T^{1,2}.$ Hence we see that $\ell$ is also a linear functional on $L^2(K).$ As a consequence of the Riesz representation theorem, we deduce that there exists a function $g_K\in L^2(K)$ such that $
\ell(f)$ is representable by $\int_K  fg_K\,dtdy$. Taking an exhaustive sequence of compacts, we obtain a function $g\in L^2_{\mathrm{loc}}(\R^{n+1}_+)$ such that $\ell(f)=\int_{\R^{n+1}_+} fg\,dtdy$ once $f\in T^{1,2}$ and $f$ has a compact support. Next, if we test the functional $\ell$ against all atoms supported in a  tent $\widehat B$ %i.e. functions $a$ such that $\mathrm{supp}(a)\subset \widehat{B}$ for some ball $B$, 
then we obtain $\|\ell\|^2\geq 1/|B|\int_{\widehat B} |g|^2 \, dtdy$, which proves the assertion. This representation is extendable to all functions in $T^{1,2}$ since compactly supported functions are dense in $T^{1,2}$.
This proves $T^{\infty,2}\supset (T^{1,2})^*.$ 
% Consider $D$, the following dense subset of $L^2$:
 %$T^{1,2} :$
%$D=\{F:\text{there exists a compact\,} K \subset \R^{n+1}_+, \mathrm{supp}(F)\subset K, F\in L^2(K)\}.$  THE PLAN: Restrict to $L^2(K),$ use Riesz representation and a standard density argument.

\medbreak
\textbf{Step 3.} To complete  the proof, it is left to show that the adjoint operator 
$$A_2^*(G)(s,\cdot)=e^{s\Delta}\int_0^\infty \nabla \mathbb{P} e^{t\Delta} G(t,\cdot) \,dt$$ 
acts from the space $T^{1,2}$ to the space $T^{1,\infty}.$
We shall use here the atomic decomposition of the space $T^{1,2}$ that follows.

\begin{lemma}
\label{atomrazl}
For any function $G\in T^{1,2}$ there exist  atoms $a_j$ and 
numbers $\lambda_j\in \mathbb C$ such that $\sum_j |\lambda_j|<\infty,$ 
satisfying $G=\sum_j \lambda_j a_j.$ On top of that, it holds that
$$\sum_j  |\lambda_j|\lesssim \|G\|_{T^{1,2}}\lesssim \sum_j |\lambda_j|.$$ 
\end{lemma}
\begin{proof}
Let us sketch the proof of this decomposition. Let $k$ be a natural number and denote $O_k:=\{x\in \R^n: S(G)(x)>2^k\}$ the corresponding level set of the function $S(G)$. Let $\mathfrak M$ denote the Hardy--Littlewood maximal function: $\mathfrak M f(x)$ is the supremum of  the averages of $\frac{1}{|B|}\int_B |f|$ taken over all  open balls $B$  that contain $x$. Consider further the sets $O_k^{*}:=\{x\in \R^n: \mathfrak M\mathbbm{1}_{O_k}(x)>1-\gamma\}$ with $\gamma$ sufficiently close to one (to be revealed in a moment). Then, as $O_k$ is open and using the weak type (1,1) of $\mathfrak M$, 
we have 
$$O_k\subset O_k^{*},\quad |O_k^{*}|\leq c(\gamma)|O_k|,\quad \widehat{O_k}\subset \widehat{O_k^{*}}$$ 
and the set $\bigcup_k \widehat{O_k^{*}}$ contains the support of the function $G$. Consider a Whitney decomposition of $O_k^{*}$: $O_k^{*}=\bigcup_j Q_j^k$ with the cubes $Q_j^k$ having the property that their diameters are comparable with the distance from $Q_j^k$ to the complement of the set $O_k^{*}.$ Next, consider a ball $B_j^{k}$ centered at the center of the cube $Q_j^{k}$ and having the radius equal to some large constant times the edge length of $Q_j^{k}.$ %in a way that $Q_j^{k}\subset B_j^{k}$. 
If this constant is large enough, then we can write the following disjoint union :
$$\Delta^k=\bigcup_j \Delta_j^k,$$
where $\Delta^k:=\widehat{O_k^{*}}-\widehat{O_{k+1}^{*}}$ and
$$\Delta_j^k:= \widehat{B_j^k}\cap ((0,\infty)\times Q_j^k)\cap \biggl(\widehat{O_k^{*}}-\widehat{O_{k+1}^{*}}\biggr)\cdotp$$
We are now in position to define the desired atomic decomposition. Denote $$a_j^k:=G\mathbbm{1}_{\Delta_j^k}|B_j^k|^{-1/2}(\mu_j^k)^{-1/2}
\quad\hbox{with }\ 
\mu_j^k:=\int_{\Delta_j^k}|G(t,y)|^2 \,dt \,dy.$$
 Set $\lambda_j^k=|B_j^k|^{1/2}(\mu_j^k)^{1/2}.$ We then have 
$$G=\sum_{k,j}\lambda_j^k a_j^k.$$
Note that the functions $a_j^k$ are atoms associated to the balls $B_j^k.$ 
\smallbreak
Recall that $\|a_j^k\|_{T^{1,2}}\lesssim 1$. Hence, $\|S(G)\|_{L^1(\R^n)}\lesssim \sum \lambda_{j}^{k}.$ Therefore, it is left to show that
\begin{equation}
\label{normatom}
\sum_{j,k}\lambda^k_j \lesssim \|S(G)\|_{L^1(\R^n)}.
\end{equation}
First of all, it is easy to see that
$$\mu_j^k=\int_{\Delta_j^k}|G(t,y)|^2 dy dt \leq \int_{\widehat{B_j^k} \cap (\widehat{O_{k+1}})^c}|G(t,y)|^2 \,dy \,dt.$$
Second of all, we shall prove that
\begin{equation}
\label{normatom1}
 \int_{\widehat{B_j^k} \cap (\widehat{O_{k+1}})^c}|G(t,y)|^2\, dy \,dt \lesssim \int_{B_j^k \cap (O_{k+1})^c}|S(G)(x)|^2\, dx.
\end{equation}
For a set $E\subset \R^n,$ introduce the notation $\Gamma(E):=\bigcup_{x\in E} \Gamma(x)$. Note that $\Gamma(E)$ is a subset of $\R^{n+1}_+$ that consists of all parabolic cones centered at points of $E.$ Apply the Fubini theorem~:
$$\begin{aligned}  
\int_{B_j^k \cap (O_{k+1})^c}|S(G)(x)|^2 \,dx&=\int_{B_j^k \cap (O_{k+1})^c}\int_{\R^{n+1}_+}\mathbbm{1}_{B(x,\sqrt{t})}(y)|f(t,y)|^2\, \frac{dy\, dt}{t^{n/2}} \,dx \\
&=\int_{\Gamma(B_j^k \cap (O_{k+1})^c)}\bigl|B(y,\sqrt{t})\cap B_j^k \cap (O_{k+1})^c\bigr|  \frac{|f(t,y)|^2 \,dy \,dt}{t^{n/2}}\cdotp
\end{aligned}$$
Inequality~\eqref{normatom1} will follow, if we show that for all $(t,y)\in \Gamma(B_j^k \cap (O_{k+1})^c),$ there holds
\begin{equation}
\label{normatom2}
t^{n/2}\lesssim \bigl|B(y,\sqrt{t})\cap B_j^k \cap (O_{k+1})^c\bigr|.
\end{equation}
Denote $F:=B_j^k \cap (O_{k+1})^c.$ Note that $(t,y)\in \Gamma(B_j^k \cap (O_{k+1})^c)$ implies that there exists $x\in B_j^k \cap (O_{k+1})^c$ such that $|x-y|\leq \sqrt{t}.$ It can be easily seen from geometric observations that there exists a universal constant  $\varepsilon<1$ depending only on $n$  such that :
$$ \bigl|B(x,\sqrt{t})\cap B(y,\sqrt{t})^c\bigr|\leq \varepsilon  \bigl|B(x,\sqrt{t})\bigr|.$$
Observe that $$(F\cap B(y,\sqrt{t})) \cup (B(x,\sqrt{t})\cap B(y,\sqrt{t}))^c \supset F\cap B(x,\sqrt{t}),$$ and hence
$$\begin{aligned}
|F\cap B(y,\sqrt{t})| &\geq |F\cap B(x,\sqrt{t})^c| - |B(y,\sqrt{t})\cap B(y,\sqrt{t})^c|\\ &\geq 
 |F\cap B(x,\sqrt{t})^c| - \varepsilon |B(x,\sqrt{t})|\\ &\geq (1-\gamma)|B(x,\sqrt{t})|- \varepsilon |B(x,\sqrt{t})|,
\end{aligned}$$
where the last inequality follows from the fact that $x\in B_j^k \subset O_k$ (and hence $\mathfrak M\mathbbm{1}_{O_k}(x)>1-\gamma$). So, the lines~\eqref{normatom1} and~\eqref{normatom2} follow if only we take $\gamma$ sufficiently close to $1$.
\smallbreak
Using the lower bound~\eqref{normatom2} and the definition of the set $O_{k+1},$ we see that 
$\mu_j^k\lesssim |Q_j^k|2^{2k}.$
With help of this inequality, we now complete the proof of the estimate~\eqref{normatom} :
 $$\begin{aligned}
\sum_{j,k}\lambda_j^k&=\sum_{j,k}|B_j^k|^{1/2}(\mu_j^k)^{1/2}\\
 &\lesssim \sum_{j,k} |B_j^k|^{1/2} |Q_j^k|^{1/2}2^{k} \\
 &\lesssim \sum_{j,k} |Q_j^k|2^{k}\\
 &\lesssim \sum_{k} |O_k^*|2^{k}\lesssim \sum_{k} |O_k|2^{k} \lesssim \|S(G)\|_{L^1(\R^n)},
\end{aligned}$$
and the lemma is proven.
\end{proof}

We shall next show that the operator 
$$\mathcal MG(\cdot):=\int_0^\infty \nabla \mathbb{P} e^{t\Delta} G(t,\cdot) \,dt,$$
which is already defined and bounded from $L^2(\mathbb R_+^{n+1}, \mathbb C^n \otimes \mathbb C^n)$ to $L^2(\R^n, \mathbb C^n \otimes \mathbb C^n)$ (see the argument below when calculating $\mathcal Ma$) can be shown to extend from the tent space $T^{1,2}$ to the Hardy space $H^{1}(\R^n, \mathbb C^n \otimes \mathbb C^n)$. %Hence it is left to check the third step on the atoms. %For an atom $a\in T^{1,2}$, we pose
%$$m:=\int_0^\infty \nabla \mathbb{P} e^{t\Delta} a(t,\cdot) dt.$$
For the reader's convenience, we add a definition and some important properties of Hardy spaces.
\begin{defin}
\label{hardy}
Let $f$ be a bounded tempered distribution on $\R^n$ and let $0<p<\infty$. Let $\Psi$ be a positive function in the Schwartz class such that $\int_{\R^n} \Psi\,dx =1$ and denote $\Psi_t(x)=t^{-n}\Psi(x/t).$
We say that $f$ lies in the Hardy space $H^p(\R^n)$ if the non-tangential maximal function 
$$M^*f(x):=\sup_{t>0}\sup_{y\in B(x,t)}|\Psi_t\ast f(y)|$$
belongs to the space $L^p(\R^n).$ The corresponding norm is given by $\|f\|_{H^p(\R^n)}:=\|M^*f\|_{L^p(\R^n)}$
\end{defin} 
\begin{remark}
\label{nezavis}
The space defined above does not depend on the particular choice of $\Psi$. The proof of this statement can be found, for instance, in the book~\cite{grafmod}, Theorem 6.4.4.
\end{remark}
\begin{remark}
The spaces $H^p(\R^n)$ coincide with $L^p(\R^n)$ for $1<p<\infty$. However for $0<p\leq 1$ this is no longer true, see~\cite{grafmod}, Theorem 6.4.3. 
\end{remark}
\begin{remark}
The dual of the space $H^1(\R^n)$ is the space $\mathrm{BMO}(\R^n)$, see~\cite{grafmod}, Theorem 6.4.3.
\end{remark}
\begin{remark}
Based on the definition of the space $H^1(\R^n)$, one can define the space $H^1(\R^n, \mathbb C^n\otimes \mathbb C^n)$ in the component-wise way.
\end{remark}

Hardy spaces also admit atomic decompositions. However, we shall not use them in this text. On the other hand, the connected notion of a Hardy molecule will be very useful in the end of the proof of the third step. These Hardy molecules should be thought of as ``sums of Hardy atoms''. In other words, a Hardy molecule can have (in comparison to a Hardy atom) an unbounded support. However, it must satisfy a ``decay at infinity'' property as we shall observe in the following definition.

\begin{defin}
\label{molecula}
Let $0<p\leq 1 < q \leq \infty$ and $b>1/p-1/q.$ Then, a $(p,q,b)$ -- molecule centered at $x_0\in \R^n$ is a real-valued function $m$ defined on $\R^n$ and satisfying :
$$|||m|||:=\|m\|_q^{1-\theta} \|\,|\cdot -x_0|^{nb} m\|_q^{\theta}<\infty,$$
where $\theta = (1/p-1/q)/b,$ (so that $0<\theta<1$) and
$$\int_{\R^n} x^{\beta} m(x) \,dx=0$$
for every multi-index $\beta$ such that $|\beta|\leq [n(1/p-1)]$.
\end{defin}

\begin{remark}
The definition above is taken from the book~\cite{garsia} (see page 328, Definition 7.13). In the very same book, it is proven that each function $m$ as in Definition~\ref{molecula} belongs to the Hardy space $H^p(\R^n)$ with the norm control $\|m\|_{H^p(\R^n)} \lesssim |||m|||$ (see Theorem 7.16 
at page 330).
\end{remark}

We go back to the proof of the third step. Let $G\in T^{1,2}$. Thanks to the atomic decomposition of the space $T^{1,2}$ in Lemma~\ref{atomrazl}, we can write
$$G=\sum_i \lambda_i a_i$$
where $a_j$ are $T^{1,2}$ atoms and $\sum_j |\lambda_j| \lesssim \|G\|_{T^{1,2}}.$ Hence
\begin{equation*}
 \sum_j |\lambda_j| \|\mathcal Ma_j\|_{H^1(\R^n,\mathbb C^n \otimes \mathbb C^n)}
\lesssim \sum_j |\lambda_j| \lesssim \|G\|_{T^{1,2}},
\end{equation*} provided we know  that $\|\mathcal Ma\|_{H^1(\R^n,\mathbb C^n \otimes \mathbb C^n)} \lesssim 1$ for any atom $a\in T^{1,2}.$ Remark at this stage that such atoms are $L^2$ functions so that $\mathcal Ma$ was already defined.
This shows that  the series  $\sum_j \lambda_j \mathcal Ma_j$ converges in  the Banach space $H^1(\R^n,\mathbb C^n \otimes \mathbb C^n)$ and we can conclude provided we identify its limit as $\mathcal MG$. It will suffice to do it for $G$ in a dense class.
% \|\mathcal MG\|_{H^1(\R^n,\mathbb C^n \otimes \mathbb C^n)}\leq
\smallbreak
We begin with the uniform estimate. Let $B$ be a ball in $\R^n$ and let $a\in T^{1,2}$ be an atom supported in the tent $\widehat{B}$. First, we need to prove the estimate
\begin{equation}
\label{atomraz}
\int_{2B}|\mathcal Ma(x)|^2\,dx\leq\int_{\R^n}|\mathcal Ma(x)|^2\,dx\lesssim \frac{1}{|B|}\cdotp
\end{equation}
Note that the Plancherel theorem yields :
$$\int_{\R^n}|\mathcal Ma(x)|^2\,dx=\frac1{(2\pi)^n}\int_{\R^n}|\mathcal F \mathcal Ma(\xi)|^2\,d\xi.$$
Hence, using the definition of $\mathcal M$ and denoting by $M_{\mathbb P}$ the matrix symbol of $\mathbb P,$  we get
$$\int_{\R^n}|\mathcal Ma(x)|^2\,dx=\frac1{(2\pi)^n} \int_{\R^n}\biggl|\int_0^\infty \xi \otimes M_{\mathbb P} (\xi) \,e^{-t|\xi|^2}{\mathcal F}a(t,\xi)\,dt\biggr|^2dx$$
whence, by using Cauchy-Schwarz inequality and the fact that  $M_{\mathbb P}$ is bounded by $1,$
$$\begin{aligned}
\int_{\R^n}|\mathcal Ma(x)|^2\,dx &\leq \frac1{(2\pi)^n}\int_{\R^n} \biggl(\int_0^\infty |\xi|^2 e^{-2t|\xi|^2}\,dt\biggr)
\biggl(\int_0^\infty |\mathcal F a(t,\xi)|^2\,dt\biggr)d\xi\\
&\leq  \frac12 \frac1{(2\pi)^n} \int_{\R^n}\int_0^\infty  |\mathcal F a(t,\xi)|^2\,dt\\
&= \frac12 \int_{\R^n}\int_0^\infty  |a(t,x)|^2\,dt\,dx \leq \frac12 |B|^{-1},\end{aligned}
$$ 
since $a$ is an atom supported in $\widehat B.$ 
Hence, the estimate~\eqref{atomraz} follows.

%Let $q_t$ denote the kernel of the operator $\nabla \mathbb{P} e^{t\Delta}$. Recall (see~\eqref{oseen}) that for all $x\in \R^n$ and $t>0,$ there holds :
%$$|q_t(x)|\lesssim \frac{1}{t^{n/2+1/2}}\left(1+\frac{|x|}{t^{1/2}}\right)^{-n-1}.$$
%From here we see that for all $\xi\in\R^n$ one has
%$$\begin{aligned}|\mathcal Fq_t(\xi)|\leq \|q_t\|_{L^1(\R^n)} &\lesssim \int_{B(0,\sqrt{t})}\frac{t \,dx}{(|x|+t^{1/2})^{n+1}} + \int_{\R^n\backslash B(0,\sqrt{t})}\frac{t \,dx}{(|x|+t^{1/2})^{n+1}} \\
%& \lesssim t^{1-n/2-1/2}(\sqrt{t})^n +  \int_{\R^n\backslash B(0,\sqrt{t})}\frac{t}{|x|^{n+1}}\,dx \lesssim t^{1/2}. \end{aligned}$$
%On the other hand, from the definition of the operator $\nabla \mathbb{P} e^{t\Delta},$ one has $$|\mathcal Fq_t(\xi)|\lesssim \frac{|\xi|e^{-\frac{|\xi|^2}{t}}}{t^{n/2}}\lesssim t^{1/2-n/2}.$$
%Taking into account the definition of the operator $\mathcal M$, using first the Cauchy--Schwarz inequality and then the bounds on the Fourier transformations of kernel $q_t$ above we deduce that
%\begin{multline*}
%\int_{\R^n}|\mathcal F \mathcal Ma(x)|^2\,dx=\int_{\R^n}\left|\int_0^\infty \mathcal F q_t \mathcal F a dt\right|^2 \,dx \leq  \int_{\R^n}\biggl(\int_0^\infty |\mathcal F q_t|^2 \biggr)\biggl(\int_0^\infty |\mathcal Fa|^2\biggr) \,dx \\
  %\lesssim 
%\int_{\R^n}\int_0^\infty |\mathcal Fa|^2\lesssim \int_{\widehat{B}}|a|^2\lesssim |B|^{-1},
%\end{multline*}
%where in the second estimate above we have used that $1-n<-1$. 

\medbreak
Second, fix $q\in(1,n/(n-1))$ and denote $b:=2(q-1)/q.$ Note that $q\leq 2$ once $n\geq 2.$ We shall show the inequality
\begin{equation}
\label{atomdva}
\int_{\R^n} |\mathcal Ma(x)|^{q} |x-x_0|^{qnb}\,dx\lesssim R^{n(qb+1-q)},
\end{equation}
where $x_0$ is the center of the ball $B$ and $R$ is its radius. As a consequence of the estimate~\eqref{atomraz}, thanks to the H\"older inequality we get the bound
\begin{equation}
\label{atomtri}
\int_{2B} |\mathcal Ma(x)|^q |x-x_0|^{qnb}\,dx \lesssim |B|^{qb}\int_{2B} |\mathcal Ma(x)|^q \,dx \lesssim R^{n(qb+1-q)}.
\end{equation}
Denote for $j\geq 1$ the dyadic annuli $\Omega_j:=2^{j+1}B\backslash 2^{j}B.$ The decay at infinity of the kernel $q_t$ and the Cauchy--Schwarz inequality yield
\begin{eqnarray}
\label{atomchetyre}
\int_{\Omega_j} |\mathcal Ma(x)|^q |x-x_0|^{qnb}\,dx &\!\!\!\leq\!\!\!& 
|B|^{qb} 2^{jqnb} \int_{\Omega_j} |\mathcal Ma(x)|^q \,dx\nonumber\\ &\!\!\!\leq\!\!\!& 
|B|^{qb} 2^{jqnb} \int_{\Omega_j} \Bigl(\int_{\widehat{B}}|a|^2\Bigr)^{q/2}\cdot \Bigl(\int_{\widehat{B}}|q_t(x-y)|^2 \,dy \,dt\Bigr)^{q/2}dx\nonumber\\ &\!\!\!\leq\!\!\!& 
 |B|^{qb-q/2}2^{jqnb}\int_{\Omega_j}\Bigl(\int_{\widehat{B}}|q_t(x-y)|^2 \,dy \,dt \Bigr)^{q/2}\, dx\nonumber\\
&\!\!\! \lesssim\!\!\!&
R^{nqb-\frac{nq}{2}}2^{jqnb}|\widehat{B}|^{\frac{q}{2}}
|\Omega_j| (2^jR)^{-q(n+1)}\nonumber\\ 
&\!\!\!=\!\!\!&R^{qnb-\frac{nq}{2}+\frac{(n+2)q}{2}+n-q(n+1)}2^{j(qnb-qn-q+n)},
\end{eqnarray}
where the last inequality above is justified by the fact that $|x-y|\lesssim (2^jR)^{-(n+1)}$ once $x\in \Omega_j$ and $(t,y)\in \widehat{B}.$ Note that 
$$qnb-nq/2+(n+2)q/2+n-q(n+1)\equiv n(qb +1-q).$$
This means that the estimates~\eqref{atomtri} and~\eqref{atomchetyre} are of the same order in $R.$  Since $q<n/(n-1)$, recalling the definition of $b$ one gets $qnb-qn-q+n=q(n-1)-n<0.$ So, the corresponding sum over integer $j\geq 0$ is convergent. Hence the estimate~\eqref{atomdva} follows. Arguing in the same way as in the estimate of the term~\eqref{atomchetyre} one gets
$$
\int_{(2B)^c} |\mathcal Ma(x)|^q \,dx \lesssim R^{n(1-q)}. 
$$
This and the second part of the inequality~\eqref{atomtri} give
\begin{equation}
\label{atompyat}
\int_{\R^n} |\mathcal Ma(x)|^q \, dx \lesssim R^{n(1-q)}.
\end{equation}

Third, observe that  $$\int_{\R^n}\mathcal Ma(x) \,dx=0.$$ This follows from the observation that $$\int_{\R^n}\mathcal Ma(x)\, dx= \int_0^\infty \int_{\R^n} \nabla \mathbb{P} e^{t\Delta}a(t,x) \,dx \,dt$$ and the divergence theorem gives $$\int_{\R^n} \nabla \mathbb{P} e^{t\Delta}a(t,x) \,dx=0.$$ This observation together with the inequalities~\eqref{atomdva} and~\eqref{atompyat} signify (by definition) that $\mathcal Ma$ is a $(1,q,b)$ matrix valued Hardy molecule (see Definition~\ref{molecula}) and in turn that $\|\mathcal Ma\|_{H^1(\R^n,\mathbb C^n\otimes \mathbb C^n)}\lesssim 1.$ 

It remains to do the identification of the series as $\mathcal M G$ for $G$ in a dense class of $T^{1,2}$. We adapt an argument of \cite{ahm} for the convenience of the reader. The subspace of those $T^{1,2}$ functions having compact support in $\mathbb R_+^{n+1}$ is dense in $T^{1,2}$. Remark that this space is also contained in $L^2(\mathbb R_+^{n+1},\mathbb C^n\otimes \mathbb C^n)$. Let $G$ be in such a space and pick as before an atomic decomposition
$G=\sum_i \lambda_i a_i$. The convergence is in $T^{1,2}$. If the series was to converge also in $L^2(\mathbb R_+^{n+1},\mathbb C^n\otimes \mathbb C^n)$ we would be done 	as $\mathcal M$ is continuous from $L^2(\mathbb R_+^{n+1}, \mathbb C^n\otimes \mathbb C^n)$  to $L^2(\mathbb R^n,\mathbb C^n\otimes \mathbb C^n)$ and the equality $\mathcal M G=\sum_i \lambda_i \mathcal M a_i $ would follow (with $L^2$ convergence). To see this equality with the appropriate interpretation, we proceed with a further truncation. For $k\in \mathbb N$, let $\chi_k(t,x)= \mathbbm 1_{B(0, 2^k)} (x) \mathbbm 1_{[2^{-k},2^k]}(t)$. Then for $k$ large enough,  $G=G\chi_k= \sum_i \lambda_i a_i\chi_k$ with convergence in $T^{1,2}$. Since $a_i$  is supported in a tent $\widehat B_i$,  if the radius  of $B_i$ is too small, then $a_i\chi_k=0$. 
Thus the radii are bounded below in this series and this, together with $T^{1,2}$ convergence, implies $L^2$ convergence. It follows that $\mathcal M G=\sum_i \lambda_i \mathcal M (a_i \chi_k)$ for $k$ large enough. The calculations above implies that $\mathcal M (a_i \chi_k)$ are uniformly bounded in $H^1$ and also that 
for $i$ fixed, $\mathcal M (a_i \chi_k) \to \mathcal M a_i$ in $H^1$ as $k\to \infty$. It easily follows that  
$\sum_i \lambda_i \mathcal M (a_i \chi_k) \to \sum_i \lambda_i \mathcal M a_i$ in $H^1$ as $k\to \infty$ and we are done.

This finishes the proof of the fact that the operator $\mathcal M$ extends boundedly from the tent space $T^{1,2}$ to the Hardy space $H^1(\R^n,\mathbb C^n\otimes \mathbb C^n).$
%and
%follow from the already discussed bound on the Oseen kernel $q_t(x)$ of the operator $e^{t\Delta}\mathbb{P}$ (see~\eqref{oseen})
%$$|k_t(x)|\leq \frac{1}{t^{n/2}}\left(1+\frac{|x|}{t^{1/2}}\right)^{-n-1}.$$
\smallbreak
We shall further prove that for a function $h$ in the Hardy space $H^1(\R^n),$  there holds that $(s,x)\rightarrow e^{s\Delta}h(x)\in T^{1,\infty}(\R^n,\mathbb C)$ with the norm estimate $$\|N(e^{s\Delta}h)\|_{L^1(\R^n)}\lesssim \|h\|_{H^1(\R^n)}.$$ Let $x\in \R^n$. Note that by definition 
$$\begin{aligned}
N(e^{s\Delta}h)(x)&= \sup_{\{(s,y) : y\in B(x,\sqrt{s})\}}\left|\int_{\R^n}\Phi_s(y-z) h(z) \,dz\right|\\
&=\sup_{\{(\xi,y) : y\in B(x,\xi)\}}\left|\int_{\R^n}\Phi_{\xi^2}(z) h(y-z) \,dz\right|,\end{aligned}$$
where $\Phi_s$ for $s>0$ stands for the heat kernel, see Definition~\ref{heat}. Since the function 
$$\Phi_{\xi^2}(z)=(4\pi \xi^2)^{-n/2}e^{-|z|^2/(4\xi^2)}$$
satisfies conditions of the definition of the Hardy space from Remark~\ref{nezavis}, the assertion  $N(e^{s\Delta}h) \in L^1(\mathbb R^n)$ follows.  Furthermore  $(s,x)\mapsto e^{s\Delta}h(x)$ is continuous on $\mathbb R_+^{n+1}$ and has the desired non-tangential almost everywhere limit by standard arguments.
The same of course holds component-wise for $\mathbb C^n\otimes \mathbb C^n$ valued Hardy functions.

So, we can conclude that the operator $A_2^*$ is bounded from $T^{1,2}$ to $T^{1,\infty}$ which means that the third step is proven. 
\end{proof}

%\begin{appendices}

%\end{appendices}

%	\cleardoublepage\markboth{Index}{Index}
	%\printindex
	%\renewcommand{\refname}{References}
	%\addbibresource{sample.bib}
	%\bibliography{sample}

\begin{thebibliography}{99}
\bibitem{pasdor}
Auscher P., Frey D., (2015), On the well-posedness of parabolic equations of Navier--Stokes type with $\mathrm{BMO}^{-1}$ data, \emph{J. Inst. Math. Jussieu} {\bf 16} (2017), no. 5, 947--985.
\bibitem{ahlmt}
Auscher P., Hofmann S., Lacey M., McIntosh A., Tchamitchian P., (2002), The solution of the Kato square root problem for second order elliptic operators on $\R^n,$
 \emph{Ann. of Math.} {\bf 156} (2002), no. 2, 633--654.
 \bibitem{ahm}
Auscher P.,  McIntosh A., Russ E. 
Hardy spaces of differential forms on Riemannian manifolds. \emph{ J. Geom. Anal.} {\bf 18} 1, (2008), 192--248 
\bibitem{amp}
Auscher P., Monniaux S., Portal P., (2012), The maximal regularity operator on tent spaces,
\emph{Commun. Pure Appl. Anal.} {\bf 11} (2012), no. 6, 2213--2219.
\bibitem{bourpav}
Bourgain J., Pavlovic N., (2008), Ill-posedness of the Navier-Stokes equations in a critical space in 3D,
\emph{J. Funct. Anal.} {\bf 255} (2008), no. 9, 2233--2247. 
\bibitem{canplan}
Cannone M., (1997), A generalization of a theorem by Kato on Navier--Stokes equations, 
\emph{Rev. Mat. Iberoamericana} {\bf 13} (1997), no. 3, 515--541.
\bibitem{cora}
Chang D., Sadosky C., Functions of bounded mean oscillation,  \emph{Taiwanese J. Math.}
{\bf 10} (2006), no. 3, 573--601.
\bibitem{cms}
Coifman R., Meyer Y., Stein E., (1985), Some New Function Spaces and Their Applications to Harmonic Analysis, \emph{J. Funct. Anal.} {\bf 62} (1985), no. 2, 304--335.
\bibitem{desimon}
De Simon L., (1964), Un'applicazione della theoria degli integrali singolari allo studio delle equazioni diverenziali lineare astratte del primo ordine, \emph{Rend. Sem. Mat. Univ. Padova} {\bf 34} (1964), 205--223.
\bibitem{dubois}
Dubois, S.,  What is a solution to the Navier-Stokes equations? \emph{C. R. Math. Acad. Sci. Paris} {\bf 335} (2002), no. 1, 27--32.
\bibitem{FK} Fujita, H.,  Kato, T.,  On the Navier-Stokes initial value problem. I. 
\emph{Arch. Rational Mech. Anal.} {\bf 16} (1964), 269--315. 
\bibitem{garsia}
Garcia-Cuerva J., Rubio de Francia J.,  \emph{Weighted Norm Inequalities and Related Topics,}
North-Holland Mathematics Studies, {\bf 116}  (1985). 
 North-Holland.
 \bibitem{GM}  Giga  Y.,  Miyakawa, T., Solutions in Lr of the Navier-Stokes initial value problem. 
 \emph{Arch. Rational Mech. Anal.} {\bf 89}  (1985), no. 3, 267--281. 
\bibitem{grafmod}
Grafakos L.,   \emph{Modern Fourier Analysis, Second edition.} Graduate Texts in Mathematics, {\bf 250} Springer, New York, 2009.
\bibitem{kato}
Kato T., (1984), Strong $L^p$ solutions of the Navier--Stokes equation in $\R^m$, with applications to weak solutions, \emph{Math. Z.} {\bf 187} (1984), no. 4, 471--480.
\bibitem{kochtat}
Koch H., Tataru D., (2001), Well-posedness for the Navier--Stokes Equations, 
\emph{Adv. Math.} {\bf 157} (2001), no. 1, 22--35. 
\bibitem{lemar}
Lemari\'e-Rieusset P-G., (2002), \emph{Recent developments in the Navier--Stokes problem,} 
 Chapman \& Hall/CRC Research Notes in Mathematics, vol. {\bf 431}, 2002.
\bibitem{stein} Stein E., (1993), Harmonic Analysis: Real-Variable Methods, Orthogonality, and Oscillatory Integrals. Princeton Mathematical Series, {\bf 43}. Monographs in Harmonic Analysis, III. Princeton University Press, Princeton, NJ, 1993.



\end{thebibliography}
%	\printbibliography[title={Bibliography}]
%\begin{comment}

%\end{comment}
\end{document}